\documentclass[reqno, 10pt]{amsart}

\usepackage{epsf,amssymb,times, graphicx, amscd, verbatim, amsxtra}
\usepackage[all]{xypic}
\usepackage[mathscr]{eucal}
\usepackage{enumerate}
\xyoption{frame}

\newtheorem{thm}{Theorem}[section]

\newtheorem{prop}[thm]{Proposition}
\newtheorem{defn}[thm]{Definition}
\newtheorem{lemma}[thm]{Lemma}
\newtheorem{cor}[thm]{Corollary}

\theoremstyle{definition}
\newtheorem{rmk}[thm]{Remark}

\newcommand{\A}{\mathbb{A}}
\newcommand{\F}{\mathbb{F}}

\newcommand{\C}{\mathbb{C}}
\newcommand{\R}{\mathbb{R}}

\newcommand{\Q}{\mathbb{Q}}
\renewcommand{\H}{\mathbb{H}}

\newcommand{\ra}{\rightarrow}

\newcommand{\Sw}{\mathscr{S}}

\newcommand{\Gm}{\mathbb{G}_m}
\newcommand{\D}{D^{\times}}
\newcommand{\JL}{\operatorname{JL}}

\newcommand{\Ind}{\operatorname{Ind}}

\newcommand{\Hom}{\operatorname{Hom}}
\newcommand{\im}{\operatorname{Im}}

\newcommand{\Irr}{\operatorname{Irr}}

\newcommand{\GSp}{\operatorname{GSp}}
\newcommand{\GSO}{\operatorname{GSO}}
\newcommand{\GO}{\operatorname{GO}}
\newcommand{\Sp}{\operatorname{Sp}}
\newcommand{\SO}{\operatorname{SO}}
\newcommand{\OO}{\operatorname{O}}
\newcommand{\GL}{\operatorname{GL}}

\newcommand{\bsigma}{\boldsymbol{\sigma}}
\newcommand{\disc}{\operatorname{disc}}
\newcommand{\M}{\operatorname{M}_{2\times2}}
\newcommand{\Id}{\operatorname{Id}}
\newcommand{\tr}{\operatorname{tr}}
\newcommand{\ch}{\operatorname{char}\,}
\renewcommand{\Re}{\operatorname{Re}}

\newcommand{\ie}{{\em i.e. }}

\begin{document}


\title[local-global non-vanishing of theta lifts]%
{Some local-global non-vanishing results of theta lifts\\ for symplectic-orthogonal dual pairs}


\author{Shuichiro Takeda}
\address{Department of Mathematics\\
Purdue University\\
150 N. University Street\\
West Lafayette, IN 47907-2067} \email{stakeda@math.purdue.edu}
\urladdr{http://www.math.purdue.edu/\char126 stakeda}


\date{\today}

\begin{abstract}
Following the approach of B. Roberts, we characterize the non-vanishing of the global theta lift for symplectic-orthogonal dual pairs in terms of its local counterpart. In particular, we replace the temperedness assumption present in Robert's work by a certain weaker assumption, and apply our results to small rank similitude groups. Among our applications is a certain instance of Langlands functorial transfer of a (non-generic) cuspidal automorphic representation of $\GSp(4)$ to $\GL(4)$.
\end{abstract}

\maketitle


\section{Introduction}\label{S:intro}


In this paper, we consider both local and global theta lifting for the symplectic-orthogonal dual pair $(\Sp(2n),\OO(V_r))$, where $V_r$ is a symmetric space of an even dimension $m$ and the Witt index $r$, with emphasis on the non-vanishing problem of global theta lifts for this pair. In particular by following the approach by B. Roberts \cite{Rob99-2}, we characterize the global non-vanishing in terms of its local counterpart. Our main theorem is

\begin{thm}\label{T:main1}
Let $\pi$ (resp. $\sigma$) be a cuspidal automorphic representation of $\Sp(2n,\A)$ (resp. $\OO(V_r,\A)$) realized in a space $V_\pi$ (resp, $V_\sigma$) of cusp forms such that at each place $v$, $\pi_v$ (resp. $\sigma_v$) is bounded by some $e_v<1$. Assume $\chi$ is the quadratic character associated with $V_r$. Then
\begin{enumerate}[(1)]
\item (From orthogonal to symplectic.) Suppose $n=\frac{1}{2}\dim V_r$. If $\sigma_v$ has a non-zero theta lift to $\Sp(2n,F_v)$ at all the places $v$ and the (incomplete) standard $L$-function $L^S(s,\sigma)$ does not vanish at $s=1$ (a pole is allowed), then the global theta lift $\Theta_n(V_\sigma)$ to $\Sp(2n,\A)$ does not vanish, and further if $L^S(s,\sigma)$ has a pole at $s=1$, then the global theta lift $\Theta_{n-1}(V_\sigma)$ to $\Sp(2n-2,\A)$ does not vanish, provided $n\geq 1$.
\item (From symplectic to orthogonal.) Suppose $\frac{1}{2}\dim V_r=n+1$. If $\pi_v$ has a non-zero theta lift to $\OO(V_r,F_v)$ at all the places $v$ and the (incomplete) standard $L$-function $L^S(s,\pi)$ does not vanish at $s=1$ (a pole is allowed), then the global theta lift $\Theta_{V_r}(V_\pi)$ to $\OO(V_r,\A)$ does not vanish, and further if $L^S(s,\pi,\chi)$ has a pole at $s=1$, then the global theta lift $\Theta_{V_{r-1}}(V_\pi)$ to $\OO(V_{r-1},\A)$ does not vanish, provided $r\geq 1$.
\end{enumerate}
\end{thm}

Here we use the notion of ``boundedness" of each local representation. To introduce this notion, let us recall that for a classical reductive group $G$ over a (not necessarily non-archimedean) local field, each irreducible admissible representation $\pi$ is the Langlands quotient of the standard module $\delta_1\times\dots\times\delta_t\rtimes\tau:=\Ind_P^G\delta_1\otimes\dots\otimes\delta_t\otimes\tau$, where each $\delta_i$ is an essentially tempered representation of some $\GL(n_i)$ and $\tau$ is a tempered representation of a lower rank group of the same type as $G$. Then for each $i$, there exists $e(\delta_i)>0$ so that $\delta_i\otimes|\;|^{-e(\delta_i)}$ is tempered, and $e(\delta_1)>\cdots>e(\delta_t)>0$. Now we make the following definition.
\begin{defn}
Let $G$ be a classical reductive group over a (not necessarily non-archimedean) local field $F$. If $\pi$ is an irreducible admissible representation which is the Langlands quotient of the standard module $\delta_1\times\dots\times\delta_t\rtimes\tau$, then we say that $\pi$ is bounded by $e$ if $e(\delta_1)\leq e$. If $t=0$, \ie $\pi$ is tempered, we say that $\pi$ is bounded by $e$ for all $e\geq 0$.
\end{defn}

Let us mention that in the above theorem if $\sigma$ is tempered, then part (1) of the theorem is essentially the main theorem of \cite{Rob99-2} by Roberts, although he also assumes some other technical assumptions for the archimedean place. (See also \cite{Takeda} for those issues.) Indeed, our proof is a modification of the one by Roberts. What forced him to impose the temperedness assumption on his theorem is that he needed his own local result in \cite{Rob98} which required the representation be tempered. The main part of this paper is to replace his temperedness assumption in \cite{Rob98, Rob99-2} with the boundedness assumption as above. The key technical point is that the ``bound" of the local representation as defined above is closely related to the size of the ``local exponents" of the representation. We will show that Roberts' arguments in \cite{Rob99-2} also work under this weaker assumption. We also consider not only the lift from the orthogonal group to the symplectic one, but also from the symplectic group to the orthogonal group as in the theorem. Indeed, the method of Roberts works for the ``symplectic-to-orthogonal" case.

Next we apply this theorem, as Roberts did, to global theta lifting for groups of similitudes of small ranks. The first case we consider is the situation of \cite{Rob01}, in which he considered the theta lift from (any form of) $\GO(V)$ with $\dim V=4$ to $\GSp(4)$. In \cite{Rob01}, he needed the temperedness assumption because the same assumption is present in \cite{Rob99-2}. We can replace his temperedness assumption by our boundedness assumption, but for the group $\GO(V)$ with $\dim V=4$ one can check that, except certain degenerate cases, the boundedness assumption is always satisfied. To state our theorem, let us recall that each cuspidal automorphic representation of $\GO(V)$ is of the form $\sigma=(\tau,\delta)$ where $\tau$ is a cuspidal representation of $\GSO(V)$ and $\delta$ is what we call an ``extension index". Then if the discriminant $d=\disc V=1$, then $\GSO(V,\A)$ is of the form $(\D(\A)\times\D(\A))/\A^\times$ where $D$ is a (possibly split) quaternion algebra and so $\tau$ is identified with a cuspidal representation $\pi_1\otimes\pi_2$ of $\D(\A)\times\D(\A)$ such that $\pi_1$ and $\pi_2$ have a same central character. In this case, we write $\tau=\tau(\pi_1,\pi_2)$. If $d\neq 1$, then $\GSO(V,\A)$ is closely related with an inner form $\GSO(V_{D,E},\A)$ of $\GL(2,\A_E)$ where $E=F(\sqrt{d})$, and $\tau$ is naturally identified with a cuspidal representation $\pi$ of this inner form whose central character is of the form $\omega\circ N_F^E$, where $N_F^E$ is the norm map from $\A_E^\times$ to $\A$. In this case, we write $\tau=\tau(\pi,\omega)$. (See section \ref{S:similitude} for the detail.) Then we prove

\begin{thm}\label{T:main2}
Let $\sigma=(\tau,\delta)$ be an infinite dimensional cuspidal unitary automorphic representation of $\GO(V,\A)$. Then
\begin{enumerate}[(1)]
\item Assume that $\tau$ is NOT of the form $\tau=\tau(\pi_1,\pi_2)$ with one of $\pi_i$ finite dimensional. Then the global theta lift of $\sigma$ to $\GSp(4)$ does not vanish if and only if each local constituent $\sigma_v$ has a non-zero theta lift to $\GSp(4)$.
\item Assume $\tau=\tau(\pi_1,\pi_2)$ with one of $\pi_i$, say $\pi_2$, finite dimensional, and so we can write $\pi_1=\chi\circ N$ where $\chi$ is a Hecke character on $\A_F^{\times}$ and $N$ is the norm map on $\D(\A)$. Then the global theta lift of $\sigma$ to $\GSp(4)$ does not vanish if each local constituent $\sigma_v$ has a non-zero theta lift to $\GSp(4)$ and the (incomplete) $L$-function $L^S(s,\pi_1\otimes\chi)$ does not vanish at $s=\frac{1}{2}$.
\end{enumerate}
\end{thm}

Let us note that part (1) of the above theorem has been already proven by Gan and Ichino in their recent preprint \cite{Gan_Ichino} by an entirely different method. However in this paper we give a proof of this theorem in our method and demonstrate that the project of Roberts in his series of papers \cite{Rob98, Rob99, Rob99-2, Rob01} can be essentially completed.

The second case we consider is global theta lifting from $\GSp(4)$ to $\GO(V_D)$ where $V_D=D\oplus\H$ with $D$ possibly split, which implies a certain instance of Langlands functorial lift. Namely we will show

\begin{thm}\label{T:main3}
Let $\pi$ be a non-generic irreducible cuspidal automorphic representation of $\GSp(4,\A)$ which satisfies the following assumptions.
\begin{enumerate}[(1)]
\item The standard degree 5 $L$-function $L^S(s,\pi)$ does not vanish at $s=1$.
\item There is a (non-complex) place $v_0$ at which $\pi_{v_0}$ has a non-zero theta lit to $\GO(V_D)$ for both split and non-split $D$.
\item At each place $v$, $\pi_v$ is bounded by $e_v<1$.
\end{enumerate}
Then there is an automorphic representation $\Pi$ on $\GL(4,\A)$ which is the strong functorial lift of $\pi$ corresponding to the map
\[
    \hat{\GSp(4)}=\GSp(4,\C)\hookrightarrow\GL(4,\C)=\hat{\GL(4)}.
\]
\end{thm}

Here we have to impose unfortunate local assumptions, especially assumption (3) in the theorem. The issue of how much those assumptions are needed will be discussed in some detail after the proof of this theorem.

As our last application, we prove the following facts, which are well-known if the base field is a totally real number field.

\begin{thm}\label{T:main4}
Let $\pi$ be a cuspidal representation of $\GSp(4, \A)$ over a (not necessarily totally real) number field $F$. Assume the incomplete standard degree $5$ $L$-function $L^S(s,\pi)$ does not vanish at $s=1$ (a pole is allowed). Then if $\pi_v$ is generic for each $v$, then $\pi$ is globally generic.
\end{thm}

\begin{cor}\label{C:main}
The multiplicity one theorem holds for the generic representations for $\GSp(4)$ over a (not necessarily totally real) number field $F$.
\end{cor}

As we mentioned, those are essentially well-known and indeed the first one is Theorem 8.1 of \cite{KRS} and the second one is the main theorem of \cite{J-S}. However in both of their works, they assumed that the base field $F$ is a totally real number field. This is because this assumption is present in \cite{KRS}. However our method implies the same result as \cite[Theorem 8.1]{KRS} without the restriction on the base field.

\quad\\

The main structure of this paper is as follows. In Section 2, we set up our notations. In Section 3, we define the notion of the bounds of exponents of irreducible admissible representations and prove certain facts related to this notion, which are necessary for our main purposes. In Section 4, we prove the key fact on the local theta correspondence, which significantly improves the main theorem of \cite{Rob98}. Then in Section 5, we prove Theorem \ref{T:main1}, and in Section 6 we prove all the theorems related to $\GSp(4)$, namely Theorems \ref{T:main2}, \ref{T:main3} and \ref{T:main4}, and Corollary \ref{C:main}.

\quad\\

\begin{center}
Acknowledgements
\end{center}

The author would like to thank Wee Teck Gan for letting him know about his work with A. Ichino \cite{Gan_Ichino}. Thanks are also due to Freydoon Shahidi for answering several questions about generic representations, and Annegret Paul for her help with the archimedean theta correspondence. Finally, he would like to thank Brooks Roberts for showing his interests in this work. Indeed a part of the remark after the proof of Theorem \ref{T:main3} is based on a conversation the author had with B. Roberts.

\quad\\


\section{Notations and Preliminaries}\label{S:notation}


In this paper,  $F$ is a local or global field of $\ch F =0$, and if $F$ is a global field we denote the ring of adeles by $\A$. If
$E$ is a quadratic extension of $F$, then we denote by $N^E_F$ (or
simply by $N$) the norm map, and by $\chi_{E/F}$ the quadratic
character obtained by local or global class field theory.

We work with smooth representations instead of $K$-finite ones.
Namely if $G$ is a reductive group over a global filed $F$, then by
a (cuspidal) automorphic form we mean a smooth (cuspidal)
automorphic form on $G(\A_F)$ in the sense of \cite[Definition
2.3]{Cogdell04}.

For a reductive group $G$ over a local field, we denote by $\Irr(G)$ the class of (equivalence classes of) irreducible admissible representations of $G(F)$. For $\pi\in\Irr(G)$, we denote the space of representation by $V_{\pi}$, though we occasionally identify the space of $\pi$ with $\pi$ itself when there is no danger of confusion. Also we denote the contragredient by $\pi^\vee$. Now assume $G$ is a classical reductive group and $P$ is a standard parabolic subgroup whose Levi is isomorphic to $\GL(n_1)\times\cdots\times\GL(n_t)\times G'$, where $G'$ is a lower rank group of the same type as $G$. Let $\delta_i$ be an admissible representation of $\GL(n_i)$ and $\tau$ an admissible representation of $G'$. Following Tadic, we write
\[
    \delta_1\times\cdots\times\delta_t\rtimes\tau:=\Ind_{P}^{G}\delta_1\otimes\cdots\otimes\delta_t\otimes\tau
    \quad\text{(normalized induction)}.
\]
Here and elsewhere, induction is always normalized. If $\delta_i$ is essentially tempered, we denote by $e(\delta_i)$ the real number so that $\delta_i\otimes|\;|^{-e(\delta_i)}$ is tempered. If all the $\delta_i$ are essentially tempered with
\[
    e(\delta_1)>\cdots>e(\delta_t)>0,
\]
and $\tau$ is tempered, we call $\delta_1\times\cdots\times\delta_t\rtimes\tau$ a standard module.

Now for $\pi\in\Irr(G)$ and a parabolic subgroup $P$, we denote by $R_{P}(\pi)$ the normalized Jacquet module of $\pi$ along $P$, and by $\overline{R}_{P}(\pi)$ the normalized Jacquet module of $\pi$ along the opposite parabolic $\overline{P}$ of $P$. It is well-known that
\[
    \overline{R}_{P}(\pi)=R_{P}(\pi^\vee)^\vee.
\]

Assume $G=\GL(n)$. For $\pi\in\Irr(G)$, we denote the central character of $\pi$ by $\omega_{\pi}$. Also we denote by $P^{\GL}_{n_1,\dots,n_t}$ the standard parabolic of $\GL(n)$ whose Levi is $\GL(n_1)\times\cdots\times\GL(n_t)$.

Let $V$ be an even dimensional symmetric space defined over a field $F$ of even dimension $m$ equipped with a symmetric bilinear form. If $V$ is defined over a local or global field $F$ of $\ch F= 0$, then we denote by $\disc V\in F^{\times}/F^{{\times}^2}$ the discriminant of $V$ when $V$ is viewed as a quadratic form. We let $\chi_V:F^{\times}\ra\{\pm1\}$ be the quadratic character of $V$, namely $\chi_V(a)=(a,(-1)^{\frac{m(m-1)}{2}}\disc V)_F$ for $a\in F^{\times}$, where $(\ ,\ )_F$ is the Hilbert symbol of $F$. Sometimes we omit $V$ and simply write $\chi$.

Also assume the Witt index of $V$ is $r$ and so $V=V_a\oplus\H^r$ where $V_a$ is anisotropic and $\H$ is the hyperbolic plane. If $\dim V=m$, we write $\dim V_a=m_a$. Now we fix a Witt decomposition
\[
    V_r=V_r'\oplus V_a\oplus V_r''
\]
with bases $\{v_1,\dots,v_r\}$ for $V_r'$ and $\{v'_1,\dots,v'_r\}$ for $V_r''$ satisfying $(v_i,v_j)=(v'_i,v'_j)=0$ and $(v_i,v_j')=\delta_{ij}$. Then we denote by $P_k$ (or sometimes $Q_k$) the parabolic subgroup of $\OO(V_r)$ that fixes $\{v_1,\dots,v_k\}$, which is a standard maximal parabolic of $\OO(V_k)$. Also if $V_r$ is split, we sometimes write $\OO(V_r)=\OO(r,r)$. Similarly for $W$ a symplectic space of rank $n$, we fix a polarization
\[
    W=W'\oplus W''
\]
with fixed symplectic bases $\{e_1,\dots,e_n\}$ for $W'$ and $\{e'_1,\dots,e'_n\}$ for $W''$. Then we denote by $P_k$ the parabolic subgroup of $\Sp(W)=\Sp(2n)$ that fixes $\{e_1,\dots,e_k\}$, which is a standard maximal parabolic of $\Sp(2n)$.

We denote the (local or global) Weil representation for
$\OO(V)\times \Sp(2n)$ by $\omega_{V,n}$ or simply by $\omega$ when
$V$ and $n$ are clear from the context. If $F$ is an archimedean local field, then the Weil representation is a smooth representation
$\omega_{V, n}$ of the group $\OO(V)\times\Sp(2n)$ of moderate growth
in the sense of \cite{Casselman89}. We say that
$\sigma\in\Irr(\OO(V))$ and $\pi\in\Irr(\Sp(2n))$ correspond, or
$\sigma$ corresponds to $\pi$ if there is a non-zero homomorphism of
Harish-Chandra modules from the underlining Harish-Chandra module of $\omega_{V, n}$ to the underlining Harish-Chandra module of
$\sigma\otimes\pi$. It is known that for each non-zero $\pi$, if a non-zero $\sigma$ corresponds to it, then $\sigma$ is unique up to infinitesimal equivalence, and hence the canonical completion of the underlining Harish-Chandra module of $\sigma$ is unique. We denote this canonical completion by $\theta_V(\pi)$. Similarly, we define $\theta_n(\sigma)$. Next assume $F$ is non-archimedean. We say that
$\sigma\in\Irr(\OO(V))$ and $\pi\in\Irr(\Sp(2n))$ correspond, or
$\sigma$ corresponds to $\pi$ if there is a non-zero
$\OO(V)\times\Sp(2n)$ homomorphism from $\omega_{V,n}$ to
$\sigma\otimes\pi$, \ie $\Hom_{\OO(V)\times\Sp(2n)}(\omega_{V,n},
\sigma\otimes\pi)\neq 0$. For an irreducible cuspidal automorphic representation $\sigma$ of $\OO(V,\A)$,
we denote by $\Theta_n(V_{\sigma})$ or $\Theta_n(\sigma)$ (when $V_{\sigma}$ is clear from the context) the space of the global theta lifts of
$\sigma$ to $\Sp(2n,\A)$, \ie the space generated by the forms of
the form $\theta(f,\phi)$ for $f\in V_{\sigma}$ and
$\phi\in\Sw(V(\A_F)^n)$ which are define by
\[
    \theta(f;\varphi)(g)
    =\int_{\OO(V, F)\backslash \OO(V,\A_F)}\Big(\sum_{x\in V(F)^n}\omega(h,g)\varphi(x)\Big)
f(h) \, dh
\]
for each $g\in\Sp(2n,\A)$. Similarly for a cuspidal automorphic representation $\pi$ of $\Sp(2n,\A)$, we denote by $\Theta_{V}(V_{\pi})$ or $\Theta_V(\pi)$ the space of the global theta lifts of $\pi$ to $\OO(V,\A)$.

\quad\\


\section{Bounds of exponents}\label{S:exponent}


In this section we prove certain facts on exponents of admissible representations, which are necessary for proving our theorems on theta lifting. First recall the following definition we made in Introduction. Namely if $\pi\in\Irr(G)$ with $G$ a classical group is the Langlands quotient of the standard module $\delta_1\times\dots\times\delta_t\rtimes\tau$, then we say that $\pi$ is bounded by $e\in\R$ if $e\geq e(\delta_1)$. If $t=0$, \ie $\pi$ is tempered, we say that $\pi$ is bounded by all $e\geq 0$. Clearly if $\pi$ is bounded by some $e$, then it is bounded by all $e'\geq e$. Note that although this notion applies to archimedean $F$, throughout this section, we assume that our field $F$ is non-archimedean.

This notion of boundedness can be shown to be closely related to the notion of exponents of $\pi$. So let us recall the notion of exponent of a representation $\pi$. Let $P$ be a parabolic subgroup of $G$ whose Levi is $M$ and $A$ the split maximal torus whose centralizer is $M$. Then an exponent $\omega$ of $\pi$ along $P$ is the restriction to $A$ of the central character of a constituent of the (normalized) Jacquet module $\overline{R}_{P}(\pi)$. For example, if $G=\Sp(2n)$ and $P=P_k$ is the standard maximal parabolic whose Levi is $\GL(k)\times\Sp(2n-2k)$, then $A$ is the center of $\GL(k)$, and so if $\pi_1\otimes\pi_2\in\Irr(\GL(k)\times\Sp(2n-2k))$ is a non-zero constituent of $\overline{R}_{P_k}(\pi)$, then the central character $\omega_{\pi_1}$ of $\pi_1$ is an exponent along $P_k$. (Note that occasionally in the literature, an exponent of $\pi$ is defined as the real number $s$ when one writes $|\omega|=|\cdot|^s$, \ie literally the ``exponent" of $\omega$. But in this paper, we simply call the character $\omega$ an exponent. We essentially follow \cite{Silberger}.) Now in this paper, we only consider an exponent along a maximal parabolic subgroup, and hence by an exponent we always mean an exponent along a maximal parabolic. Then let us make the following definition.

\begin{defn}
Let $\pi$ be an admissible representation of $\Sp(2n)$ (resp. of $\OO(V_r)$). We say that the exponents of $\pi$ are bounded by $e\geq 0$, if for all exponents $\omega$ along the standard maximal parabolic subgroup $P_k$  for all $1\leq k\leq n$, we have
\[
    |\omega(a)|\leq|a|^{ke}
\]
for $a\in F^{\times}$ with $|a|>1$. Also if $\overline{R}_{P_k}(\pi)=0$ for all $P_k$, \ie $\pi$ is supercuspidal, we say that the exponents are bounded by any $e\geq 0$.
\end{defn}

Let us note that for each admissible representation $\pi$ of finite length, there are only finitely many exponents along each parabolic, and hence for each such $\pi$ the exponents are bounded by some non-negative number. Also clearly if the exponents of $\pi$ are bounded by some $e$, then they are also bounded by all $e'\geq e$.

For tempered representations, the exponents are bounded by $0$. Namely,
\begin{prop}\label{P:Silberger}
\begin{enumerate}[(a)]
\item Let $\pi$ be a non-supercuspidal admissible tempered representation of $\Sp(2n)$ (resp. $\OO(V_r)$) over $F$ which is of finite length. Then for all standard maximal parabolic $P_k$ whose Levi is $\GL(k)\times\Sp(2n-2k)$ (resp. $\GL(k)\times\OO(V_{r-k})$) and every exponent $\omega$ along $P_k$, we have
\[
    |\omega(a)|\leq 1
\]
for $a\in F^{\times}$ with $|a|>1$. In particular, the exponents of $\pi$ are bounded by $0$.
\item Let $\sigma$ be a non-supercuspidal admissible representation of $\GL(n)$ over $F$ which is essentially tempered and of finite length. Then for all standard maximal parabolic $P^{\GL}_{k,n-k}$ whose Levi is $\GL(k)\times\GL(n-k)$, let $\sigma_1\otimes\sigma_2\in\Irr(\GL(k)\times\GL(n-k))$ be a non-zero constituent of $\overline{R}_{P^{\GL}_{k, n-k}}$. Then
\[
    \left|\frac{\omega_{\sigma_1}(a_1)}{\omega_{\sigma_2}(a_2)}\right|\leq \left|\frac{{a_1}^k}{{a_2}^{n-k}}\right|^{e(\sigma)}
\]
for $a_1,a_2\in F^{\times}$ with $\left|\frac{a_1}{a_2}\right|>1$, where recall that $e(\sigma)$ is the real number such that $\sigma\otimes|\;|^{-e(\sigma)}$ is tempered.
\end{enumerate}
\end{prop}
\begin{proof}
This easily follows from \cite[Corollary 2.6]{Silberger}.
\end{proof}

Note that by combining this proposition with our convention that for a supercuspidal representation the exponents are bounded by any $e\geq 0$, we can simply say that for a tempered representation the exponents are bounded by any $e\geq 0$.

Now the following gives the relation between the bound of $\pi$ and its exponents, and this is the key proposition for our computation of local theta correspondences.

\begin{prop}\label{P:exponent}
Let $\pi$ be an irreducible admissible representation of $\Sp(2n)$ (resp. $\OO(V_r)$) which is a non-zero constituent of the standard module
\[
    \delta_1\times\cdots\times\delta_t\rtimes\tau.
\]
Then the exponents of $\pi$ are bounded by $e(\delta_1)$. In particular, if $\pi$ is the Langlands quotient and so is bounded by $e(\delta_1)$, then the exponents of $\pi$ are bounded by $e(\delta_1)$.
\end{prop}

To prove the proposition, we need the following lemma.

\begin{lemma}
Let $\pi$ be an irreducible admissible representation of $\Sp(2n)$ (resp. $\OO(V_r)$) which is a non-zero constituent of
\[
    \sigma\rtimes\tau,
\]
where $\sigma$ is an essentially tempered representation of $\GL(l)$ with $e(\sigma)>0$, and $\tau$ is an admissible representation of $\Sp(2n-2l)$ (resp. $\OO(V_{r-l})$) of finite length in which the exponents of every constituent are bounded by $e(\sigma)$. Then the exponents of $\pi$ are bounded by $e(\sigma)$.
\end{lemma}
\begin{proof}
Let us first treat the case for the symplectic group $\Sp(2n)$. Let $P_k$ be the standard maximal parabolic with the Levi $\GL(k)\times\Sp(2n-2k)$. Recall $\overline{R}_{P_k}(\pi)=R_{P_k}(\pi^\vee)^\vee$, and let $\pi_1\otimes\pi_2$ be a constituent of $\overline{R}_{P_k}(\pi)$. We first compute the Jacquet module of
\[
    \Ind_{P_{l}}^{\Sp(2n)}(\sigma\otimes\tau)^\vee=\Ind_{P_{l}}^{\Sp(2n)}(\sigma^\vee\otimes\tau^\vee)
\]
along $P_k$. For this, let $l=l_1+l_2+l_3$ be a partition of $l$ and $n-l=m_1+m_2$ a partition of $n-l$ such that
\[
    l_3+l_1+m_1=k,
\]
where we allow some of $l_i$ and $m_i$ to be zero. By \cite[Theorem 5.4]{Tadic95} each constituent of the Jacquet module of $\Ind_{P_{l}}^{\Sp(2n)}(\sigma^\vee\otimes\tau^\vee)$ along $P_k$ is a constituent of a representation of the form
\[
    \Ind_{P^{\GL}_{l_3,l_1,m_1}}^{\GL(k)}(\sigma_3^\vee\otimes\sigma_1\otimes\tau_1)\otimes
    \Ind_{P'_{l_2}}^{\Sp(2n-2k)}(\sigma_2\otimes\tau_2),
\]
where $\sigma_1\otimes\sigma_2\otimes\sigma_3$ is a constituent of the Jacquet module of $\sigma^\vee$ along the standard parabolic $P_{l_1,l_2,l_3}^{\GL}$ whose Levi is $\GL(l_1)\times\GL(l_2)\times\GL(l_3)$, and $\tau_1\otimes\tau_2$ is a constituent of the Jacquet module of $\tau^\vee$ along the standard parabolic whose Levi is $\GL(m_1)\times\Sp(2m_2)$. Hence the central character $\omega_{\pi_1}$ of $\pi_1$ is
\[
    \omega_{\pi_1}=\omega_{\sigma_3}\omega_{\sigma_1}^{-1}\omega_{\tau_1}^{-1}.
\]

Since $\sigma$ is essentially tempered and $\sigma_1^\vee\otimes\sigma_2^\vee\otimes\sigma_3^\vee$ is a constituent of $R_{P_{l_1,l_2,l_3}^{\GL}}(\sigma^\vee)^\vee=\overline{R}_{P_{l_1,l_2,l_3}^{\GL}}(\sigma)$, by Proposition \ref{P:Silberger}, for $|a|>1$ we have
\[
    |\omega_{\sigma_1}^{-1}(a)|\leq|a|^{l_1e(\sigma)}
\]
and
\[
    |\omega_{\sigma_3}(a)|\leq|a|^{-l_3e(\sigma)}.
\]
Also since the exponents of $\tau$ are bounded by $e(\sigma)$ by our assumption and $\tau_1^\vee\otimes\tau_2^\vee$ is a constituent of $R_{P_{m_1}}(\tau^\vee)^\vee=\overline{R}_{P_{m_1}}(\tau)$, we have
\[
    |\omega_{\tau_1}^{-1}(a)|<|a|^{m_1e(\sigma)}.
\]
Hence for $|a|>1$, we have
\[
    |\omega_{\pi_1}(a)|<|a|^{(l_1-l_3+m_1)e(\sigma)}.
\]
Notice that since $k=l_3+l_1+m_1$,
\[
    (l_1-l_3+m_1)e(\sigma)=(k-2l_3)e(\sigma)\leq ke(\sigma).
\]
So the lemma is proven for the symplectic group.

For the orthogonal group, the proof is identical except that we need to use \cite{Ban} instead of \cite{Tadic95} for the computation of the Jacquet module of induced representations. Note that in \cite{Ban} she treats only the split orthogonal group, but her proof works for the non-split group. Indeed, the proof of the non-split case is even simpler. The detail is left to the reader.
\end{proof}

Now we are ready to prove Proposition \ref{P:exponent}.

\begin{proof}[Proof of Proposition \ref{P:exponent}]
We will prove it by induction on $t$, \ie the number of inducing data. Assume $t=1$ and so $\pi$ is a constituent of $\delta_1\rtimes\tau$. By Proposition \ref{P:Silberger} (a), the exponents of $\tau$ are bounded by $0$ and so \emph{a fortiori} by $e(\delta_1)$. So by the lemma the exponents of $\pi$ are bounded by $e(\delta_1)$. Now assume $\pi$ is a constituent of $\delta_1\times\cdots\times\delta_{t+1}\rtimes\tau$. Note that
\[
    \delta_1\times\cdots\times\delta_{t+1}\rtimes\tau
    =\delta_1\rtimes(\delta_2\times\cdots\times\delta_{t+1}\rtimes\tau),
\]
and so by the induction hypothesis the exponents of every constituent of $\delta_2\times\cdots\times\delta_{t+1}\rtimes\tau$ are bounded by $e(\delta_2)$ and hence by $e(\delta_1)$. The above lemma immediately implies that the exponents of $\pi$ are bounded by $e(\delta_1)$.
\end{proof}

There are a couple of corollaries we need to mention.

\begin{cor}\label{C:exponent}
If $\pi$ is a constituent of $\delta_1\times\dots\times\delta_t\rtimes\tau$ as in the proposition, then the exponents of $\pi^\vee$ are also bounded by $e(\delta_1)$.
\end{cor}
\begin{proof}
Note that $\pi^\vee$ is a constituent of $\delta_1^\vee\times\dots\times\delta_t^\vee\rtimes\tau^\vee$. But $\delta_1^\vee\times\dots\times\delta_t^\vee\rtimes\tau^\vee=\delta_1\times\dots\times\delta_t\rtimes\tau^\vee$ up to semisimplification, and if $\tau$ is tempered, then so is $\tau^\vee$. So the lemma immediately follows from the proposition.
\end{proof}

Notice that the above proposition tells us that one can tell how large the exponents of every irreducible admissible representation are by looking at the Langlands quotient data. Moreover, the Langlands quotient data gives the lowest bound for the exponents. Namely,

\begin{cor}
Assume that $\pi$ is an irreducible admissible representation of $\Sp(2n)$ (resp. $\OO(V_r)$), which is the Langlands quotient of $\delta_1\times\dots\delta_t\rtimes\tau$ with $e(\delta_1)>\cdots>e(\delta_t)>0$. Then $e(\delta_1)$ is the smallest number by which the exponents of $\pi$ are bounded.
\end{cor}
\begin{proof}
Since
\[
     \Hom_G(\delta_1\times\dots\times\delta_t\rtimes\tau, \pi)\neq 0,
\]
where $G$ is $\Sp(2n)$ (resp. ($\OO(V_r)$), by Frobenius reciprocity,
\[
     \Hom_M(\delta_1\otimes(\delta_2\times\dots\times\delta_t\rtimes\tau), \overline{R}_{P_{n_1}}(\pi))\neq 0,
\]
where $M$ is $GL(n_1)\times\Sp(2(n-n_1))$ (resp. $GL(n_1)\times\OO(V_{r-n_1})$). So there exists an exponent $\omega$ along this parabolic such that $\omega=\omega_{\delta_1}$, and hence $|\omega(a)|=|a|^{n_1e(\delta_1)}$ for $a\in F^\times$. Thus if the exponents of $\pi$ are bounded by $e$, we must have $e(\delta_1)\leq e$. But by the proposition, we already know that the exponents of $\pi$ are bounded by $e(\delta_1)$, which completes the proof.
\end{proof}

\quad\\


\section{on the local theta lift for isometry groups}\label{S:isometry}


The major object of this section is to improve upon the result of Roberts \cite{Rob98} regarding the non-archimedean theta correspondence by applying our notion of ``boundedness". In particular we will replace his temperedness assumption in \cite{Rob98} by a weaker assumption of ``bounded by 1''. Accordingly, except Corollary \ref{C:main_local}, which will be given at the end of this section, we assume that the base field is a non-archimedean local field of characteristic 0.

The first theorem in this section is
\begin{thm}
Recall that $\dim V_r=m=m_a+2r$. Let $\pi\in\Ind(\Sp(2n))$ and $\pi'\in\Ind(\Sp(2n'))$, and also let $\sigma\in\Ind(\OO(V_r))$ and $\sigma'\in\Ind(\OO(V_{r'}))$.
\begin{enumerate}[(1)]
\item (From orthogonal to symplectic.) Let $n$ and $n'$ be such that $2n'> 2n\geq \dim V_r$. Assume both $\pi$ and $\pi'$ correspond to $\sigma$ under the theta correspondence. If $\pi$ is bounded by $e\leq1$, then $\pi'$ is an irreducible quotient of
\[
    \chi|\;|^{n'-m/2}\times\chi|\;|^{n'-1-m/2}\times\cdots\times\chi|\;|^{n+1-m/2}\rtimes\pi.
\]
In particular if $\pi$ is the Langlands quotient of
\[
    \delta_1\times\cdots\times\delta_1\rtimes\tau
\]
such that $e(\delta_1)<1$, then $\pi'$ is the Langlands quotient of
\[
    \chi|\;|^{n'-m/2}\times\chi|\;|^{n'-1-m/2}\times\cdots\times\chi|\;|^{n+1-m/2}\times\delta_1\times\cdots\times\delta_1\rtimes\tau.
\]

\item (From symplectic to orthogonal.) Let $r$ and $r'$ be such that $m_a+2r'> m_a+2r\geq 2n+2$. Assume both $\sigma$ and $\sigma'$ correspond to $\pi$ under the theta correspondence. If $\sigma$ is bounded by $e\leq1$, then $\sigma'$ is an irreducible quotient of
\[
    |\;|^{m_a/2+r'-1-n}\times|\;|^{m_a/2+r'-2-n}\times\cdots\times|\;|^{m_a/2+r-n}\rtimes\sigma.
\]
In particular if $\sigma$ is the Langlands quotient of
\[
    \delta_1\times\cdots\times\delta_1\rtimes\tau
\]
such that $e(\delta_1)<1$, then $\sigma'$ is the Langlands quotient of
\[
    |\;|^{m_a/2+r'-1-n}\times|\;|^{m_a/2+r'-2-n}\times\cdots\times|\;|^{m_a/2+r-n}\times\delta_1\times\cdots\times\delta_1\rtimes\tau.
\]
\end{enumerate}
\end{thm}
\begin{proof}
The proof is essentially a modification of the one by Roberts \cite[Theorem 4.4]{Rob98}, but we use our ``bounded by $1$" assumption instead of his temperedness assumption. In what follows, we give a self-contained proof for the ``from symplectic to orthogonal" case because this is the case which is not treated by Roberts. The ``from orthogonal to symplectic" case is a straightforward modification of Roberts' proof and left to the reader.

First let
\[
    \delta=\begin{pmatrix} I_n&0\\0&-I_n\end{pmatrix}\in\GL(2n).
\]
Notice that if $g\in\Sp(2n)$, then $\delta g\delta\in\Sp(2n)$. We define $\pi^\delta$ by $\pi^\delta(g)=\pi(\delta g\delta)$. Then we have
\[
    \pi^\delta\cong\pi^\vee.
\]
(See, for example, \cite[Theorem 1.6, Ch. VI]{Kudla_notes}). Similarly, we define $\omega_{V_r,n}^\delta$ by $\omega_{V_r,n}^\delta(h,g)=\omega_{V_r,n}(h,\delta g\delta)$. Then we have
\[
    \omega_{V_r,n}^\delta\cong\omega_{-V_r,n},
\]
where $-V_r$ is the negative of the quadratic form given by $V_r$. Now since $\sigma$ and $\pi$ correspond via theta correspondence, there is a non-zero $\OO(V_r)\times\Sp(2n)$ map
\[
    \omega_{V_r,n}\rightarrow\sigma\otimes\pi.
\]
Hence by twisting by $\delta$ we have a non-zero $\OO(V_r)\times\Sp(2n)$ map
\[
    \omega_{V_r,n}^\delta\cong\omega_{-V_r,n}\rightarrow\sigma\otimes\pi^\vee.
\]
Also since $\pi$ and $\sigma'$ correspond via theta correspondence, there is a non-zero $\OO(V_{r'})\times\Sp(2n)$ map
\[
    \omega_{V_{r'},n}\rightarrow\sigma'\otimes\pi.
\]
Thus we have a non-zero $\OO(V_{r'})\times\OO(V_r)\times\Sp(2n)$ map
\[
    \omega_{V_{r'},n}\otimes\omega_{-V_r,n}\rightarrow\sigma'\otimes\sigma\otimes\pi\otimes\pi^\vee.
\]
Note that there is a canonical embedding $\OO(V_{r'})\times\OO(V_r)\hookrightarrow\OO(V_{r'}\oplus-V_r)$, where for $V_{r'}\oplus-V_r$ the bilinear form is defined by $\langle v'_1+v_1,v'_2+v_2\rangle=\langle v'_1,v'_2\rangle_{V_{r'}}-\langle v_1,v_2\rangle_{V_{r}}$. (This embedding might as well be called the embedding of the ``generalized doubling method''.) Then $\OO(V_r\oplus-V_r)$ is the split orthogonal group of rank
\[
    R=m_a+r+r'.
\]
Via this embedding, one can see that
\[
    \omega_{V_{r'},n}\otimes\omega_{-V_r,n}\cong\omega_{V_{r'}\oplus -V_r,n}|_{\OO(V_{r'})\times\OO(V_r)\times\Sp(2n)}.
\]
Now by composing the canonical $\Sp(2n)$ map $\pi\otimes\pi^\vee\rightarrow \mathbf{1}$, we obtain a surjective $\OO(V_{r'})\times\OO(V_r)$ map
\[
    (\omega_{V_{r'}\oplus -V_r,n})_{\Sp(2n)}\rightarrow\sigma'\otimes\sigma.
\]
Recall that for the orthogonal group every irreducible admissible representation is selfdual and so ${\sigma'}^\vee\otimes\sigma^\vee\cong\sigma'\otimes\sigma$. (See \cite[Theorem 1.6, Ch. VI]{Kudla_notes}.) Hence by taking the contragredient we obtain an injective $\OO(V_{r'})\times\OO(V_r)$ map
\[
    \sigma'\otimes\sigma\hookrightarrow(\omega_{V_{r'}\oplus -V_r,n})^\vee_{\Sp(2n)}\cong(\omega_{-V_{r'}\oplus V_r,n})_{\Sp(2n)}.
\]
Let $I_R(s)$ be the degenerate principal series for the split orthogonal group $\OO(-V_{r'}\oplus V_r)=\OO(R,R)$. By Theorem 5.1 of \cite[Ch. II]{Kudla_notes}, there exists an injective $\OO(R,R)$ map
\[
    (\omega_{-V_{r'}\oplus V_r,n})_{\Sp(2n)}\hookrightarrow I_R(-s_0),
\]
where
\[
    s_0=\frac{R-1}{2}-n.
\]
This gives an injective $\OO(-V_{r'})\times\OO(V_r)=\OO(V_{r'})\times\OO(V_r)$ map
\[
    \sigma'\otimes\sigma\hookrightarrow I_R(-s_0),
\]
By taking the contragredient, we have a non-zero $\OO(V_{r'})\times\OO(V_r)$ map
\[
    I_R(s_0)\rightarrow \sigma'\otimes\sigma,
\]
The degenerate principal series $I_R(s_0)$ admits a filtration of $\OO(V_{r'})\times\OO(V_r)$ representations
\[
    0=I^{r+1}\subset I^{r}\subset\cdots\subset I^1 \subset I^0= I_R(s_0)
\]
such that
\[
    I^i/I^{i+1}\cong\Ind_{P_{r'-i}\times P_{r-i}}^{\OO(V_{r'})\times\OO(V_r)}|\;|^{s_0+\frac{r-i}{2}}\otimes|\;|^{s_0+\frac{r'-i}{2}}\otimes\rho_i
\]
where $\rho_i$ is the representation of $\OO(V_i)\times\OO(V_i)$ on $\Sw(\OO(V_i))$ defined by $\rho_i(h,h')\phi(x)=\phi(h^{-1}xh')$. (This can be shown following the proof of Proposition 2.3 \cite{Kudla_notes} which shows the symplectic case.) So for some $0\leq i\leq r$, we have
\[
    \Hom_{\OO(V_{r'})\times\OO(V_r)}
    (\Ind_{P_{r'-i}\times P_{r-i}}^{\OO(V_{r'})\times\OO(V_r)}|\;|^{s_0+\frac{r-i}{2}}\otimes|\;|^{s_0+\frac{r'-i}{2}}\otimes\rho_i,
    \sigma'\otimes\sigma)\neq 0.
\]
By Frobenius reciprocity,
\[
    \Hom_{M_{r'-i}\times M_{r-i}}(|\;|^{s_0+\frac{r-i}{2}}\otimes|\;|^{s_0+\frac{r'-i}{2}}\otimes\rho_i,
    \overline{R}_{P_{r'-i}}(\sigma')\otimes\overline{R}_{P_{r-i}}(\sigma))\neq 0.
\]
So there exist an irreducible subquotient $\sigma_1\otimes\sigma_2\in\Irr(\GL(r-i)\times\OO(V_{i}))$ of $\overline{R}_{P_{r-i}}(\sigma)$ and an irreducible subquotient $\sigma'_1\otimes\sigma'_2\in\Irr(\GL(r-i)\times\OO(V_{i}))$ of $\overline{R}_{P_{r'-i}}(\sigma')$ such that
\[
    \Hom_{M_{r'-i}\times M_{r-i}}(|\;|^{s_0+\frac{r-i}{2}}\otimes|\;|^{s_0+\frac{r'-i}{2}}\otimes\rho_i,
    \sigma'_1\otimes\sigma'_2\otimes\sigma_1\otimes\sigma_2)\neq 0.
\]
Hence we have
\[
    |\;|^{s_0+\frac{r'-i}{2}}=\sigma_1.
\]

First suppose $i<r$ and so $r-i\geq 1$. By our assumption $\sigma$ is bounded by $e\leq 1$ and so for $|a|>1$, we have
\[
    |\det(a I_{r-i})|^{s_0+\frac{r'-i}{2}}=|\omega_{\sigma_1}(a)|\leq |a|^{(r-i)e}\leq|a|^{r-i},
\]
which gives
\[
    (r-i)(s_0+\frac{r'-i}{2})\leq r-i.
\]
Since we are assuming $r-i>0$, we have $s_0+\frac{r'-i}{2}\leq 1$. Recalling $s_0=\frac{R-1}{2}-n$ and $R=m_a+r+r'$, we obtain $m_a+r+r'-1-2n+r'-i\leq 2$. Also by our assumption, $2n+2\leq m_a+2r$. So $r-i\leq -2(r'-r)+1$. Now $r'-r> 0$ and so we have $r-i<1$, which is a contradiction because $r-i\geq 1$.

Hence we have $i=r$ and so
\[
    \Hom_{\OO(V_{r'})\times\OO(V_r)}
    (\Ind_{P_{r'-r}\times \OO(V_r)}^{\OO(V_{r'})\times\OO(V_r)}|\;|^{s_0}\otimes\rho_r,
    \sigma'\otimes\sigma)\neq 0.
\]
By Frobenius reciprocity, this gives
\[
    \Hom_{P_{r'-r}\times \OO(V_r)}(|\;|^{s_0}\otimes\rho_r,
    (\sigma'|_{P_{r'-r}}\otimes\delta_{P_{r'-r}}^{1/2})\otimes\sigma)\neq 0,
\]
\ie
\[
    \Hom_{P_{r'-r}\times \OO(V_r)}(\rho_r,
    (\sigma'|_{P_{r'-r}}\otimes|\;|^{-s_0}\delta_{P_{r'-r}}^{1/2})\otimes\sigma^\vee)\neq 0.
\]
Here we used the fact that for the orthogonal group every irreducible admissible representation is selfdual. Also notice that on $\rho_r$, the $\GL(r'-r)$ part of the parabolic $P_{r'-r}$ acts trivially and hence it also acts trivially on $\sigma'|_{P_{r'-r}}\otimes|\;|^{-s_0}\delta_{P_{r'-r}}^{1/2}$. Now let $f$ be a non-trivial element in this Hom space. By the lemma on p. 59 of \cite{MVW}, there is a surjective $\OO(V_r)\times\OO(V_r)$ map
\[
   \rho_r\rightarrow\sigma\otimes\sigma^\vee
\]
whose kernel is contained in $\ker f$. Hence we have a non-zero $\OO(V_r)\times\OO(V_r)$ map
\[
    \sigma\otimes\sigma^\vee\rightarrow(\sigma'|_{P_{r'-r}}\otimes|\;|^{-s_0}\delta_{P_{r'-r}}^{1/2})\otimes\sigma^\vee.
\]
Since $\GL(r'-r)$ acts trivially on $\sigma'|_{P_{r'-r}}\otimes|\;|^{-s_0}\delta_{P_{r'-r}}^{1/2}$, this map can be extended to a $\OO(V_r)\times P_{r'-r}$ map. Hence
\[
     \Hom_{P_{r'-r}}(\sigma,\sigma'|_{P_{r'-r}}\otimes|\;|^{-s_0}\delta_{P_{r'-r}}^{1/2})\neq 0.
\]
By Frobenius reciprocity,
\[
     \Hom_{\OO(V_{r'})}(\Ind_{P_{r'-r}}^{\OO(V_{r'})}(|\;|^{s_0}\boxtimes\sigma),\sigma')\neq 0.
\]
One can easily see that there is a surjective $\GL(r'-r)$ map
\[
     \Ind_{P_{1,\dots,1}^{\GL}}^{\GL(r'-r)}(|\;|^{m_a/2+r'-1-n}\times|\;|^{m_a/2+r'-2-n}\times\cdots\times|\;|^{m_a/2+r-n})\rightarrow|\;|^{s_0}.
\]
(Indeed the one dimensional representation $|\;|^{s_0}$ of $\GL(r'-r)$ is the Langlands quotient of this induced representation.) Therefore there is a surjective $\OO(V_{r'})$ map
\[
    |\;|^{m_a/2+r'-1-n}\times|\;|^{m_a/2+r'-2-n}\times\cdots\times|\;|^{m_a/2+r-n}\rtimes\sigma\rightarrow\sigma',
\]
which completes the proof.
\end{proof}

\begin{rmk}
As we mentioned at the beginning of the proof, the first case of the theorem, \ie the orthogonal-to-symplectic case is even a more straightforward modification of \cite[Theorem 4.4]{Rob98}. However in \cite{Rob98} Roberts assumes that $\sigma$ is unitarizable, which he calls pre-unitary. He needed this assumption to have a $\C$-anti-linear map $\sigma\rightarrow\sigma^\vee$. But it is possible to get around this subtle issue in the following way. First notice that $\overline{\omega_{V_r,n}}\cong\omega_{V_r,n}^\delta$, where $\overline{\omega_{V_r,n}}$ is as in the proof of Theorem 4.4. in \cite{Rob98} and $\delta$ is as in our proof above. So since $\sigma^\vee=\sigma$ and $\pi^\vee=\pi^\delta$, we have a non-zero $\OO(V_r)\times\Sp(2n)$ map
\[
    \overline{\omega_{V_r,n}}\rightarrow\sigma^\vee\otimes\pi^\vee.
\]
Hence there is a non-zero $\OO(V_r)\times\Sp(2n)\times\Sp(2n')$ map
\[
    \overline{\omega_{V_r,n}}\otimes\omega_{V_r,n'}\rightarrow\sigma^\vee\otimes\sigma\otimes\pi^\vee\otimes\pi'.
\]
Via the inclusion $\Sp(2n)\times\Sp(2n')\hookrightarrow\Sp(2(n+n'))$ given in p. 1116 of \cite{Rob98},
we have
\[
    \overline{\omega_{V_r,n}}\otimes\omega_{V_r,n'}\cong\omega_{V_r,n+n'}|_{\OO(V_r)\times\left(\Sp(2n)\times\Sp(2n')\right)}.
\]
So by composing with the canonical $\OO(V_r)$ map $\sigma^\vee\otimes\sigma\rightarrow\mathbf{1}$, we obtain a non-zero $\OO(V_r)\times\Sp(2n)\times\Sp(2n')$ map
\[
    \omega_{V_r,n+n'}\rightarrow\pi^\vee\otimes\pi'.
\]
This way, we can suppress the unitarizability assumption present in \cite{Rob98}. The reader can easily verify that the rest of the proof by Roberts does not require $\sigma$ be unitarizable.
\end{rmk}

Now a natural question to ask is for which class of $\sigma\in\Irr(\OO(V_r))$ (resp. $\pi\in\Irr(\Sp(2n))$) the corresponding $\pi\in\Irr(\Sp(2n))$ (resp. $\sigma\in\Irr(\OO(V_r))$) has the desired boundedness property as in the assumption of the theorem. The tempered version of this question has been settled by Roberts \cite[Theorem 4.2]{Rob98} for the orthogonal-to-symplectic case for the range $n\leq \dim V_r/2$. And he has shown that for this range if $\sigma$ is tempered, then $\pi$ is also tempered. We will show an analogue of this in terms of our notion of boundedness. Namely, we have

\begin{thm}\label{T:local2}
Suppose that $\pi\in\Irr(\Sp(2n))$ and $\sigma\in\Irr(\OO(V_r))$, and $\pi$ and $\sigma$ correspond under the theta correspondence. Then
\begin{enumerate}[(1)]
\item (From orthogonal to symplectic.) Assume $2n\leq\dim V_r$. If $\sigma$ is bounded by $e\geq 0$, then $\pi$ is also bounded by $e$.
\item (From symplectic to orthogonal.) Assume $\dim V_r\leq 2n+2$. If $\pi$ is bounded by $e\geq 0$, then $\sigma$ is also bounded by $e$.
\end{enumerate}
\end{thm}

The proof is by now a standard argument using Frobenius reciprocity and the Jacquet module of the Weil representation. Indeed, the proof is again almost identical to the one given by Roberts \cite[p.1113-1114]{Rob98}, especially for the orthogonal-to-symplectic case. But since this theorem is quite crucial for our global applications, we give a self-contained proof with our modified assumption for the symplectic-to-orthogonal case. First we need the following well-known result due to Kudla \cite{Kudla86}.

\begin{lemma}
Let $Q_j$ be our choice of the maximal parabolic of $\OO(V_r)$ whose Levi is $\GL(j)\times\OO(V_{r-j})$, and also let $l=\min(n,j)$. For each $0\leq k\leq l$ define $\bsigma_k$ to be the representation of $\GL(k)\times\GL(k)$ on the space $\Sw(\GL(k))$ with the action given by $\bsigma_k(g,g')\varphi(x)=\varphi(g^{-1}xg')$. Then the Jacquet module $R_{Q_j}(\omega_{V_r,n})$ of the Weil representation has a filtration
\[
    0=F^{l+1}\subset F^{l}\subset F^{l-1}\subset\cdots\subset F^1\subset F^0=R_{Q_j}(\omega_{V_r,n})
\]
such that
\[
    F^{k}/F^{k+1}\cong\Ind_{Q_{jk}\times P_k}^{(\GL(j)\times\OO(V_{r-j}))\times\Sp(2n)}\xi_{jk}\bsigma_{k}\otimes\omega_{V_{r-j},n-k},
\]
where $Q_{jk}$ is the standard parabolic of $\GL(j)$ whose Levi is $\GL(j)\times\GL(j-k)$, $P_k$ is the standard maximal parabolic of $\Sp(2n)$ whose Levi is $\GL(j)\times\Sp(2n-2k)$, and $\xi_{jk}$ is a character on $Q_{jk}\times P_k$ whose restriction on $(\GL(k)\times\GL(j-k))\times\GL(k)\subset Q_{jk}\times P_k$ is given by
\[
    \xi_{jk}(h_1,h_2,g)=|\det h_1|^{-(\frac{1}{2}m-j+\frac{k-1}{2})}|\det h_2|^{-(\frac{1}{2}m-\frac{1}{2}j-n+\frac{k-1}{2})}
    \chi(\det g)|\det g|^{\frac{1}{2}m-j+\frac{k-1}{2}},
\]
where $(h_1,h_2)\in\GL(k)\times\GL(j-k)\subset Q_{jk}$ and $g\in \GL(k)\subset P_k$.
\end{lemma}
\begin{proof}
This is essentially Theorem 2.8 of \cite{Kudla86}. However we should mention that his choice of the parabolic for $\OO(V_r)$ is different from ours, and this is the reason we have slightly different actions for $\xi_{jk}$. and $\bsigma_k$. Our convention follows \cite{Rob98}. Also a detailed proof for the case of similitude groups appears in \cite{GT4}.
\end{proof}

Using this, we prove our theorem.

\begin{proof}[Proof of Theorem \ref{T:local2}]
Once again, we only prove (2) of the theorem. So assume $\pi\in\Irr(\Sp(2n))$ is bounded by $e\geq 0$. Also assume $\sigma\in\Irr(\OO(V_r))$ corresponding to $\pi$ is the Langlands quotient of $\delta_1\times\cdots\times\delta_t\rtimes\tau$ where $\delta_1$ is a representation of $\GL(j)$. Then we will show that $e(\delta_1)\leq e$, which will prove the theorem. First note that $\sigma$ is a subrepresentation of $\delta_1^\vee\times\cdots\times\delta_t^\vee\rtimes\tau^\vee$. Let
\[
    \rho=\delta_1^\vee\quad\text{and}\quad\rho'=\delta_2^\vee\times\cdots\times\delta_t^\vee\rtimes\tau^\vee
\]
so that $\sigma$ is a subrepresentation of $\rho\rtimes\rho'$. Since $\pi$ and $\sigma$ correspond, we have
\[
    \Hom_{\OO(V_r)\times\Sp(2n)}(\omega_{V_r,n},\;(\rho\rtimes\rho')\otimes\pi)\neq 0.
\]
By Frobenius reciprocity, we have
\[
    \Hom_{(\GL(j)\times\OO(V_{r-j}))\times\Sp(2n)}(R_{Q_j}(\omega_{V_r,n}),\;(\rho\otimes\rho')\otimes\pi)\neq 0.
\]
So this Hom space is non-zero at some stage of the filtration of $R_{Q_j}(\omega_{V_r,n})$ given in the above lemma. Hence for some $0\leq k\leq l$, we have
\[
    \Hom_{(\GL(j)\times\OO(V_{r-j}))\times\Sp(2n)}(\Ind_{Q_{jk}\times P_k}^{(\GL(j)\times\OO(V_{r-j}))\times\Sp(2n)}\xi_{jk}\bsigma_{k}\otimes\omega_{V_{r-j},n-k},\;(\rho\otimes\rho')\otimes\pi)\neq 0.
\]
Hence by Frobenius reciprocity, there is a non-zero $(\GL(k)\times\GL(j-k)\times\OO(V_{r-j}))\times(\GL(k)\times\Sp(2n-2k))$ map
\[
    \xi_{jk}\bsigma_{k}\otimes\omega_{V_{r-j},n-k}\rightarrow (\overline{R}_{P^{\GL}_{k,j-k}}(\rho)\otimes\rho')\otimes\overline{R}_{P_k}(\pi).
\]
So there exist an irreducible constituent $\rho_1\otimes\rho_2\in\Irr(\GL(k)\times\GL(j-k))$ of $\overline{R}_{P^{\GL}_{k,j-k}}(\rho)$ and an irreducible constituent of $\pi_1\otimes\pi_2\in\Irr(\GL(k)\times\Sp(2n-2k))$ of $\overline{R}_{P_k}(\pi)$ so that there is a non-zero $(\GL(k)\times\GL(j-k)\times\OO(V_{r-j}))\times(\GL(k)\times\Sp(2n-2k))$ map
\[
    \xi_{jk}\bsigma_{k}\otimes\omega_{V_{r-j},n-k}\rightarrow (\rho_1\otimes\rho_2\otimes\rho')\otimes(\pi_1\otimes\pi_2).
\]
We see that
\[
    \Hom_{\GL(k)\times\GL(k)}(\xi_{jk}\bsigma_k,\rho_1\otimes\pi_1)\neq 0
\]
and
\[
    \Hom_{\GL(j-k)}(\xi_{jk},\rho_2)\neq 0.
\]
Now since the elements of the form $(aI_k,aI_k)\in\GL(k)\times\GL(k)$ acts trivially on $\bsigma_k$, from the above lemma, one sees that $\xi_{jk}(aI_k,aI_k)=\chi(a)^{k}$. Hence we must have
\[
    |\omega_{\rho_1}(a)||\omega_{\pi_1}(a)|=1.
\]
Also for $aI_{j-k}\in\GL(j-k)$, one sees that $\xi_{jk}(aI_{j-k})=|\det(aI_{j-k})|^{-(\frac{m-j}{2}-n+\frac{k-1}{2})}$. Hence
\[
    \omega_{\rho_2}=|a|^{-(j-k)(\frac{m-j}{2}-n+\frac{k-1}{2})}.
\]
Since $\rho_2\otimes\rho_1$ is a constituent of $\overline{R}_{P^{\GL}_{k,j-k}}(\rho)$, we have $\omega_{\rho}(a)=\omega_{\rho_1}(a)\omega_{\rho_2}(a)$. Hence
\[
    |\omega_{\pi_1}(a)|=|\omega_{\rho_1}(a)|^{-1}=|\omega_{\rho_2}(a)||\omega_{\rho}(a)|^{-1}.
\]
Recall that $\rho=\delta_1^\vee$ and so $|\omega_{\rho}(a)|=-je(\delta_1)$. Therefore we obtain
\[
    |\omega_{\pi_1}(a)|=|a|^{-(j-k)(\frac{m-j}{2}-n+\frac{k-1}{2})+je(\delta_1)}.
\]
Now by our assumption $\pi$ is bounded by $e$, and so for $|a|>1$, we have $|\omega_{\pi_1}(a)|\leq|a|^{ke}$, which gives
\[
    -(j-k)(\frac{m-j}{2}-n+\frac{k-1}{2})+je(\delta_1)\leq ke.
\]

Now by our assumption $m=\dim V_r\leq 2n+2$. So in the above inequality, we see that 
\[
    -(j-k)(\frac{m-j}{2}-n+\frac{k-1}{2})\geq 0. 
\]
Hence we have $0\leq -je(\delta_1)+ke$. Assume for the sake of contradiction that $e(\delta_1)>e$. If $k>0$, we would have $0<-(j-k)e(\delta_1)$, which is a contradiction because $-(j-k)e(\delta_1)\leq 0$. And if $k=0$, then we would have $0\leq -je(\delta_1)$, which is also a contradiction because $j\geq 1$ and so $je(\delta_1)>e\geq 0$. This completes the proof.
\end{proof}

\begin{rmk}
Let us mention that Roberts considers the range $2n>\dim V_r$ for the orthogonal-to-symplectic case, and has computed the first component of the Langlands quotient data of $\pi$ to some degree of explicitness, when $\pi$ is a first occurrence but not tempered. (See Theorem 4.2(2) of \cite{Rob98} for the details.) An analogous result can be shown in our situation. But we will leave this issue to the reader because the theorem will not play any role for our global applications.
\end{rmk}

By combining the two theorems above, we have the following, which will be the key result for our global applications.

\begin{cor}\label{C:main_local}
Assume $\pi\in\Ind(\Sp(2n))$ and $\sigma\in\Ind(\OO(V_r))$. Here we assume that $F$ is not necessarily non-archimedean.
\begin{enumerate}[(1)]
\item (From orthogonal to symplectic.) Let $n=\frac{1}{2}\dim V_r$. Also let $\pi'\in\Irr(\Sp(2n'))$ for $n'> n$. Assume both $\pi$ and $\pi'$ correspond to $\sigma$ under the theta correspondence. If $\sigma$ is bounded by $e<1$, then $\pi'$ is the Langlands quotient of
\[
    \chi|\;|^{n'-m/2}\times\chi|\;|^{n'-1-m/2}\times\cdots\times\chi|\;|^{1}\rtimes L(\pi),
\]
where $L(\pi)$ is the Langlands quotient data of $\pi$.

\item (From symplectic to orthogonal.) Let $\frac{1}{2}\dim V_r=n+1$. Also let $\sigma'\in\Irr(\OO(V_{r'}))$ for $r'> r$. Assume both $\sigma$ and $\sigma'$ correspond to $\pi$ under the theta correspondence. If $\pi$ is bounded by $e<1$, then $\sigma'$ is the Langlands quotient of
\[
    |\;|^{m_a/2+r'-1-n}\times|\;|^{m_a/2+r'-2-n}\times\cdots\times|\;|^{1}\rtimes L(\sigma),
\]
where $L(\sigma)$ is the Langlands quotient data of $\sigma$.
\end{enumerate}
\end{cor}
\begin{proof}
The non-archimedean case immediately follows from the two theorems in this section. For $F=\R$, this is Theorem 6.2 of \cite{Paul}. For $F=\C$, this can be read off from Theorem 2.8 of \cite{Ad_Bar95}, although the reader has to be careful in that the complex valuation in \cite{Ad_Bar95} is the usual absolute value \ie $|z|=\sqrt{z\bar{z}}$, but we use the standard complex valuation \ie $|z|=z\bar{z}$.
\end{proof}

Let us emphasize here that this corollary tells us that for the range as in the corollary the theta lift of a representation which is bounded by $e<1$ is uniquely determined. The proof never used the Howe duality of any sort, and hence the statement holds for the even residual characteristic case.

\quad\\


\section{on the global theta lift for isometry groups}\label{S:isometry}


The theorems in the previous section, especially Corollary \ref{C:main_local}, combined with the method of Roberts \cite{Rob99-2} allows us to prove our main non-vanishing theorem for global theta lifts, namely Theorem \ref{T:main1}.

\begin{proof}[Proof of Theorem \ref{T:main1}]
The proof is a modification of the beautiful argument by Roberts \cite{Rob99-2}, which in turn has its origin \cite{BS91} in the classical context. We reproduce essential points of the proof for the symplectic-to-orthogonal case. So assume $\pi\cong\otimes_v\pi_v$ is an irreducible cuspidal automorphic representation of $\Sp(2n,\A)$ which satisfies the assumption of part (2) of the theorem. Also assume that $\pi$ is realized in a space $V_\pi$ of cusp forms. First of all, the Euler product of the (incomplete) standard $L$-function $L^S(s,\pi,\chi)$ twisted by $\chi$ converges absolutely for $\Re(s)\geq 2$. This can be seen as follows. Since each $\pi_v$ is bounded by some $e_v<1$, each unramified factor of $L^S(s,\pi,\chi)$ is a product of factors of the form
\[
    (1-a_vq_v^{-1})^{-1}\quad \text{with}\quad |a_v|\leq q_v^{e_v}.
\]
One sees that the Euler product of such factors converges absolutely for $\Re(s)>1+e_v$. (See Lemma 2 in p.187 and the first paragraph of p.188 of \cite{Marcus}.) Hence it converges absolutely for $\Re(s)\geq 2$, and in particular does not vanish for $\Re(s)\geq 2$.

Now for the sake of contradiction, assume that the global theta lift $\Theta_{V_r}(V_\pi)$ is zero. Then the global theta lift $\Theta_{V_{r'}}(V_\pi)$ must be non-zero cuspidal for some $r'$ with $r<r'\leq 2n$. Let $\tau$ be an irreducible constituent of $\Theta_{V_{r'}}(V_\pi)$. By the functoriality of the unramified theta correspondence, the standard $L$-function $L^S(s,\tau)$ of $\tau$ is written as
\begin{align*}
L^S(s,\tau)&=L^S(s,\pi,\chi)\zeta^S(s+r'-r)\zeta^S(s+r'-r-1)\cdots\zeta^S(s+1)\cdot\\
&\qquad\quad\zeta^s(s)\zeta^S(s-1)\cdots\zeta^S(s-r'+r+1)\zeta^S(s-r'+r),
\end{align*}
where $\zeta^S(s)$ is the (incomplete) Dedekind zeta function. (See \cite[Corollary 7.1.4]{Kudla-Rallis94}.) Now consider the order of the zero at $s=r'-r$ for this $L$-function. One sees that $\zeta^S(s-r'+r)$ has a zero of order $|S|-1$ and $\zeta^S(s-r'+r+1)$ has a simple pole. Also $L^S(s,\pi,\chi)$ does not vanish (a pole is allowed) at $s=r'-r$. (This is true by our assumption on $\pi$ if $r'-r=1$, and by the absolute convergence if $r'-r\geq 2$.) So we see that the order of vanishing of $L^S(s,\tau)$ is at most $|S|-2$.

On the other hand, we can compute the very same $L$-function $L^S(s,\tau)$ by a totally different method, namely by the doubling method, which gives
\[
    L^S(s,\tau)=\frac{b^S(s-\frac{1}{2})Z(s-\frac{1}{2},f,\Phi)}{\prod_{v\in S}Z_v(s-\frac{1}{2},f_v,\Phi_v)},
\]
where $Z(s-\frac{1}{2},f,\Phi)$ is the doubling integral with $f=\otimes_vf_v$ a matrix coefficient of $\tau$ and $\Phi=\otimes_v\Phi_v$ a $K$-finite standard section, and $b^S(s)$ is a certain normalizing factor. Now the following lemma due to Roberts \cite{Rob99-2} is the key technical ingredient.
\begin{lemma}
Let $v$ be archimedean or non-archimedean. Assume $\tau_v$ is the Langlands quotient of $\delta_1\times\cdots\times\delta_t\rtimes\rho$ with $\delta_1=|\;|^k$ for some $k>0$. Then $f_v$ and $\Phi_v$ can be chosen so that the local zeta integral $Z_v(s-\frac{1}{2},f_v,\Phi_v)$ has a pole at $s=k$.
\end{lemma}
\begin{proof}
This is essentially the ``main lemma" of \cite{Rob99-2}. Although Roberts considers the zeta integral for the symplectic groups, as he mentions in the introduction of \cite{Rob99-2} it is straightforward to modify his computations to obtain the same result for the orthogonal groups. Also see \cite{Takeda} for a certain subtle issue for the archimedean place.
\end{proof}

Now we know that $\tau_v$ is the local theta lift of $\pi_v^\vee$. Then since $\pi_v$ has a non-zero theta lift to $\OO(V_r)$, so does $\pi_v^\vee$. Also if $\pi_v$ is bounded by $e_v<1$, so is $\pi_v^\vee$ by Corollary \ref{C:exponent}. Hence if $\tau_v$ is the Langlands quotient of $\delta_1\times\cdots\times\delta_t\rtimes\rho$, then $\delta_1=|\;|^{r'-r}$ by Corollary \ref{C:main_local}. So by the above lemma, we can choose $f_v$ and $\Phi_v$ so that $Z_v(s-\frac{1}{2},f_v,\Phi_v)$ has a pole at $s=r'-r$ for all $v\in S$. It is well-known that the normalized zeta integral $b^S(s-\frac{1}{2})Z(s-\frac{1}{2},f,\Phi)$ has at most a simple pole at $s=r'-r$ (this follows from \cite[Proposition 1.1.4]{Kudla-Rallis90}), and so we conclude that $L^S(s,\tau)$ has a zero of order at least $|S|-1$ at $s=r'-r$.

Therefore the first computation of $L^S(s,\tau)$ shows that it has a zero of order at most $|S|-2$ at $s=r'-r$, and the second computation shows that it has a zero of order at least $|S|-1$ at the same $s=r'-r$, which is a contradiction. Hence $\Theta_r(V_\pi)\neq 0$.

Now assume further that $L^S(s,\pi)$ has a pole at $s=1$, but for the sake of contradiction assume that the theta lift $\Theta_{r-1}(V_\pi)$ to $\OO(V_{r-1},\A)$ vanishes. Then by what we have shown we know that the theta lift $\Theta_{r}(V_\pi)$ to $\OO(V_r,\A)$ is non-zero cuspidal. Hence if $\tau$ is an irreducible constituent, we have
\[
    L^S(s,\tau)=\zeta^S(s)L^S(s,\pi, \chi).
\]
Since $\zeta^S(s)$ has a pole at $s=1$, if $L^S(s,\pi, \chi)$ has a pole at $s=1$, then $L^S(s,\tau)$ would have at least a double pole at $s=1$. But it is known that the poles of the standard $L$-function of the orthogonal groups are at most simple \cite[Theorem 2.0.1]{Kudla-Rallis90}. Hence $\Theta_r(\pi)\neq 0$.
\end{proof}

\begin{rmk}
Let us take this opportunity to point out a mistake in the previous work \cite{Takeda} by the author. In the non-vanishing theorem \cite[Theorem 1.1]{Takeda} we claimed that the temperedness assumption for the archimedean place can be removed from the theorem by Roberts \cite{Rob99-2}. But this is a mistake. This is due to my misunderstanding on what is meant by ``Langlands quotient" in \cite{Paul} and \cite{Ad_Bar95}. For example, for a standard module $\delta_1\times\cdots\delta_t\rtimes\tau$ with $e(\delta_1)=1$, consider the induced representation $|\;|\times\delta_1\times\cdots\delta_t\rtimes\tau$. This is not a standard module anymore, but it does have a unique irreducible quotient. And in \cite{Paul} and \cite{Ad_Bar95}, this quotient is also called the Langlands quotient. However, this quotient is not given as the image of the desired intertwining operator discussed in detail in \cite{Takeda}. Rather in our sense of ``Langlands quotient", this unique quotient is the Langlands quotient of $\delta'_1\times\cdots\times\delta_t\rtimes\tau$, where $\delta'_1=\Ind_{P^{\GL}_{1,n_1}}^{\GL(n_1+1)}|\;|\otimes\delta_1$. Hence we cannot apply Roberts' computation in \cite{Rob99-2} to this case. Indeed, we need to exclude cases like this. But otherwise our argument in \cite{Takeda} works, and then instead of the boundedness assumption as in the above theorem, it is sufficient to assume that, say for the orthogonal-to-symplectic case, for the archimedean $v$, if $\sigma_v$ is the Langlands quotient of $\delta_1\times\cdots\delta_t\rtimes\tau$, then
\[
    \{e(\delta_1),e(\delta_2),\dots,e(\delta_t)\}\cap\{1,2,\dots,\frac{1}{2}\dim V_r\}=\emptyset.
\]
Similarly, we can also replace the boundedness assumption by an analogous assumption for the symplectic-to-orthogonal case. However, to do this does not seem to have any merit but simply to complicate the statement of the theorem, and so we state our theorem keeping the assumption for the archimedean place same as that of the non-archimedean one.
\end{rmk}

\quad\\


\section{on the global theta lift for similitude groups}\label{S:similitude}

In this section, we apply our non-vanishing result in the previous section to small rank similitude groups. In particular we consider the theta lift from $\GO(V)$ with $\dim V=4$ to $\GSp(4)$, and the one from $\GSp(4)$ to $\GO(V_D)$ with $V_D=D\oplus\H$ where $D$ is a (possibly split) quaternion algebra over $F$ and $\H$ is the hyperbolic place. Hence in this section we consider similitude theta lifting. References for the theory of similitude theta lifting are abound by now (\cite{HST,Rob01,GT1} and citations therein), and so rather than repeating the detail of this theory here, we will refer the reader to those references.

Now first let us note the following lemma, which will be necessary for our applications.

\begin{lemma}\label{L:bound_restriction}
Let $(G, H)$ be the pair $(\GSp(2n), \Sp(2n))$ (resp. $(\GO(V_r), \OO(V_r))$), and  $\hat{\pi}\in\Irr(G)$ be an irreducible admissible representation of $G$ which is the Langlands quotient of the standard module $\delta_1\times\cdots\times\delta_t\rtimes\tau$. Then every constituent $\pi$ of the restriction $\hat{\pi}|_{H}$ is bounded by $e(\delta_1)$.
\end{lemma}
\begin{proof}
Let $H'$ be $\Sp(2(n-n_1-\cdots-n_t))$ (resp. $\OO(V_{r-n_1-\cdots-n_t})$). Also let us write $\tau|_{H'}=\tau_1\oplus\cdots\oplus\tau_k$, where each $\tau_i$ is tempered. Then we have $(\delta_1\times\cdots\times\delta_t\rtimes\tau)|_{H}=\bigoplus_{i=1}^{k}\delta_1\times\cdots\times\delta_t\rtimes\tau_i$. Hence if $\pi$ is a constituent of $\hat{\pi}|_{H}$, it is a quotient of some $\delta_1\times\cdots\times\delta_t\rtimes\tau_i$. So the lemma follows.
\end{proof}


\subsection{The theta lift from $\GO(4)$}

We consider the theta lift from $\GO(V)$ with $\dim V=4$. The main object here is to remove the temperedness assumption present in the work by Roberts \cite{Rob01}. However we have to mention that such result has been already proven by Gan and Ichino in their recent preprint \cite{Gan_Ichino} by an entirely different method. But we will demonstrate our method, which is a direct generalization of \cite{Rob01}, also works.

We start with classifying representations on $\GO(V)$. The proofs are found in \cite{Rob01}, \cite{HST}, \cite{Takeda} or citations
therein.\\

\noindent\underline{\textbf{The local case:}} Let us first consider the local case and so $F$ will be a local field and all the groups will be over $F$. Let $d=\disc V$ be the discriminant of $V$. The groups $\GSO(V)$ depend on $d$ as follows.
\begin{enumerate}[(1)]
\item If $d=1$, then $\GSO(V)\cong \GSO(D)$, where $D$ is a (possibly split) quaternion algebra over $F$
made into a quadratic form in the usual way. So $D$ is either the unique division quaternion algebra or the space $M_2$ of $2\times 2$ matrices. Then there is a natural
bijection between $\Irr(\GSO(V),\omega)$ and the set of irreducible
admissible representations $\pi_1\otimes\pi_2$ of $D^{\times}\times D^{\times}$ such that
both $\pi_1$ and $\pi_2$ have the same central character $\omega$.
In this case, we write $\tau=\tau(\pi_1,\pi_2)$.

\item If $d\neq 1$, then $\GSO(V)\cong\GSO(V_E)$, where $V_E=\{x\in\M(E)| \ ^cx^t=x\}$
is the space of Hermitian matrices over $E=F(\sqrt{d})$ with the
quadratic form given by $-\det$. Then there is a natural bijection
between $\Irr(\GSO(V),\omega)$ and $\Irr(\GL(2,E),\omega\circ N^E_F)$.
In this case, we write $\tau=\tau(\pi,\omega)$.
\end{enumerate}

Notice that $\GO(V)\cong\GSO(V)\rtimes\{1,t\}$, where we choose $t$ to act on $V$ as the matrix transpose if $V=M_2$ or $V=X_E$ and the quaternion conjugation if $V=D$. For each $\tau\in\Irr(\GSO(V))$, we define $\tau^c$ by taking $V_{\tau^c}=V_{\tau}$ and by letting $\tau^c(g)f=\tau(tgt)f$ for all
$g\in \GSO(V)$ and $f\in V_{\tau}$.

Then we have
\begin{enumerate}[$\bullet$]
\item If $\tau\ncong\tau^c$, then $\Ind_{\GSO(V)}^{\GO(V)}\tau$ is irreducible,
and we denote it by $\tau^{+}$.

\item If $\tau\cong\tau^c$, then $\Ind_{\GSO(V)}^{\GO(V)}\tau$ is reducible. Indeed,
it is the sum of two irreducible representations, and we write
$\Ind_{\GSO(V)}^{\GO(V)}\tau\cong\tau^{+}\oplus\tau^{-}$. Here, $t$
acts on $\tau^{\pm}$ via a linear operator $\theta^{\pm}$ with the
property that $(\theta^{\pm})^2=\Id$ and $\theta^{\pm}\circ
g=tgt\circ\theta^{\pm}$ for all $g\in \GSO(V)$.
\end{enumerate}

We can be more explicit about the irreducible components $\tau^{+}$
and $\tau^{-}$. First assume $d=1$. In this case, it is easy to see
that, via $\rho$, $t$ acts on $\GL(2)\times\GL(2)$ or
$D^{\times}\times D^{\times}$ by $t\cdot(g_1,g_2)=(g_2,g_1)$, and if
$\tau=\tau(\pi_1,\pi_2)$ is such that $\tau\cong\tau^c$, then
$\pi_1\cong\pi_2$. We can choose $\theta^{\pm}$ to be such
that $\theta^{+}(x_1\otimes x_2)=x_2\otimes x_1$ and
$\theta^{-}(x_1\otimes x_2)=-x_2\otimes x_1$ for $x_1\otimes x_2\in
{\pi_1\otimes\pi_2}$. We choose $\tau^{+}$ and $\tau^{-}$
accordingly. Note that our choice of $\tau^+$ is the \emph{preferred} extension defined in
\cite[p.227]{PSP}.

Next assume $d\neq1$. In this case $t$ acts, via
$\rho$, on $\GL(2,E)$ in such a way that $t\cdot g=\ ^cg$ \ie the Galois conjugation. If
$\tau=\tau(\pi,\omega)$ is such that $\tau\cong\tau^c$, then
$\pi\cong\pi^c$. Note that $\pi$ has a unique Whittaker model,
namely it is realized as a space of functions $f:GL(2,E)\ra\C$ such
that $f\left(\left(\begin{smallmatrix}1&a\\0&1\end{smallmatrix}\right)\right)=\psi_v(\tr\,
a)f(g)$ for all $a\in E$ and $g\in\GL(2,E)$, where $\psi_v$ is a
fixed additive character of $F$. Then we define $\theta^{\pm}$ to be
the linear operator that acts on this space of Whittaker functions
by $f\mapsto \pm f\circ c$, and $\theta^{+}$ is chosen to be the one
that acts as $f\mapsto f\circ c$. We choose $\tau^{+}$ and $\tau^{-}$
accordingly.

We should note that our choice of $\tau^{+}$ and $\tau^{-}$ is
different from that of Roberts in \cite{Rob01}, but rather we follow
\cite{HST}. The reason we make this choice is because it is consistent with Proposition \ref{P:global_disconnected} below. However as we will show at the end of this subsection (Proposition \ref{P:local_extension}), it turns out that those two choices indeed coincide. Also the reader should notice that in the above discussion the fields $F$ and $E$ do not have to be non-archimedean.\\

\noindent\underline{\textbf{The global case:}} Now we consider the global case, and hence here we assume that $F$ is a global field of $\ch F =0$ and the groups are over $F$. If $d\neq 1$ and
$E=F(\sqrt{d})$, then let $c$ be the non-trivial element in
$\text{Gal}(E/F)$. For each quaternion algebra $D$ over $F$, let
$B_{D,E}=D\otimes E$. For each $g\in B_{D,E}$, we define
$^cg^{\ast}$ by linearly extending the operation $^c(x\otimes
a)^{\ast}= x^{\ast}\otimes\, ^ca$ where $^c$ is the Galois
conjugation and $^{\ast}$ is the quaternion conjugation. The
space $V_{D,E}=\{g\in B_{D,E} :\, ^cg^{\ast}=g\}$ can be made into a
four dimensional quadratic space over $F$ via the reduced norm of
the quaternion algebra $B_{D,E}$. Similarly to the local case, we have

\begin{enumerate}[(1)]
\item If $d=1$, then $\GO(V)$ is isomorphic to $\GO(D)$ for some (possibly split)
quaternion algebra over $F$. Then there is a natural bijective
correspondence between an irreducible cuspidal automorphic representation $\tau$
of $\GSO(V,\A_F)$ whose central character is
$\omega$ and an irreducible cuspidal automorphic representation
$\pi_1\otimes\pi_2$ of $\D (\A_F)\times \D(\A_F)$ such that both
$\pi_1$ and $\pi_2$ have the central character $\omega$. In this
case, we write $\tau=\tau(\pi_1,\pi_2)$.

\item If $d\neq 1$, then there exists a quaternion algebra $D$ over $F$ such
that $\GO(V)\cong \GO(V_{D,E})$. Then there is a natural bijective
correspondence between an irreducible cuspidal automorphic representation $\tau$
of $\GSO(V, \A_F) \cong\GSO(V_{D,E}, \A_F)$ whose central character
is $\omega$ and an irreducible cuspidal automorphic representation $\pi$ of
$B_{D,E}^{\times}(\A_E)$ whose central character is of the form
$\omega\circ N^E_F$. In this case, we write $\tau=\tau(\pi,\omega)$.
\end{enumerate}

We need to consider the relation between irreducible cuspidal automorphic
representations of the two groups $\GSO(V,\A_F)$ and $\GO(V,\A_F)$. Recall that as an algebraic group, $\GO(V)\cong\GSO(V)\rtimes\{1,t\}$.
First define $\tau^c$ by taking $V_{\tau^c} = \{f\circ c : f\in
V_{\tau}\}$, where $c:\GSO(V, \A_F)\ra \GSO(V, \A_F)$ is the
isomorphism given by conjugation $g\mapsto tgt$, where $t\in\GO(V,F)\backslash\GSO(V,F)$. Then clearly
$\tau^c$ is an irreducible cuspidal automorphic representation of $\GSO(V, \A_F)$.
(Note that as an admissible representation, $\tau^c$ is isomorphic to
the representation $\tau'$ with $V_{\tau'}=V_{\tau^c}$ and the action
defined by $\tau'(g)f=\tau(tgt)f$, and so if we write
$\tau\cong\otimes\tau_v$, then $\tau^c\cong\otimes\tau_v^c$.) By
multiplicity one theorem, $\tau\cong\tau^c$ implies
$V_{\tau}=V_{\tau^c}$ and in this case $f\circ c\in V_{\tau}$. Also let
$\sigma$ be an irreducible cuspidal automorphic representation of $\GO(V, \A_F)$.
Define $V_{\sigma}^{\circ}=\{f|_{\GSO(V, \A_F)} : f\in
V_{\sigma}\}$. Then either $V_{\sigma}^{\circ}=V_{\tau}$ for some irreducible
cuspidal automorphic representation $\tau$ of $\GSO(V, \A_F)$ such
that $\tau=\tau^c$, or $V_{\sigma}^{\circ}=V_{\tau}\oplus V_{\tau^c}$
for some irreducible cuspidal automorphic representation $\tau$ of $\GSO(V,
\A_F)$ such that $\tau\neq\tau^c$. (See \cite[p.381--382]{HST}.) Then we have

\begin{prop}\label{P:global_disconnected}
Define $\widehat{\sigma}$ to be the sum of all the irreducible cuspidal automorphic
representations of $\GO(V, \A_F)$ lying above $\tau$, \ie
$\widehat{\sigma}=\oplus_i\sigma_i$ where $\sigma_i$ runs over all the irreducible
cuspidal automorphic representations of $\GO(V, \A_F)$ such that
$V_{\sigma_i}^{\circ}=V_{\tau}$ if $\tau=\tau^c$, or $V_{\sigma_i}^{\circ}=V_{\tau}\oplus V_{\tau^c}$ otherwise.
Then
\[
    \widehat{\sigma}\cong\bigoplus_{\delta}\underset{v}{\otimes}\tau_v^{\delta(v)},
\]
where $\delta$ runs over all the maps from the set of all places of
$F$ to $\{\pm\}$ with the property that $\delta(v)=+$ for almost all
places of $F$, and $\delta(v)=+$ if $\tau_v\ncong\tau_v^c$, and
further if $\tau\cong\tau^c$, then $\underset{v}{\prod}\delta(v)=+$.
Moreover each $\otimes\tau_v^{\delta(v)}$ is (isomorphic to) an irreducible
cuspidal automorphic representation of $\GO(V,\A_F)$.
\end{prop}
\begin{proof}
This is Proposition 5.4 of \cite{Takeda}.
\end{proof}

This proposition tells us that if $\tau$ is an irreducible cuspidal automorphic
representation of $\GSO(V, \A_F)$ and  $\delta$ is a map from the set of all places of
$F$ to $\{\pm\}$ having the property described in the above
proposition, then there is an irreducible cuspidal automorphic representation
$\sigma=(\tau, \delta)$ of $\GO(V, \A_F)$ lying above $\tau$ such that
$\sigma\cong\otimes\tau_v^{\delta(v)}$. We call such a map $\delta$ an
``extension index'' of $\tau$, and $(\tau, \delta)$ the extension of
$\tau$ with an extension index $\delta$.

Once we have this classification of representations of $\GO(V)$, the content of Theorem \ref{T:main2} can be understood to the reader. To prove Theorem \ref{T:main2}, we need the following lemma.
\begin{lemma}
Let $\sigma_v$ be an infinite dimensional unitary irreducible admissible representation of $\GO(V_v)$ with $\dim\sigma_v>1$, where $\dim V_v=4$. Then $\sigma_v$ is always bounded by some $e_v<1$. Here $v$ is not necessarily non-archimedean.
\end{lemma}
\begin{proof}
In this proof, we suppress the subscript $v$ and simply write $\sigma$, $V$, etc. First note that since $\sigma$ is infinite dimensional, $V$ can not be anisotropic. Assume $d=1$. Then $\sigma=\tau^\delta$ where $\delta=+$ or $-$ and $\tau=\tau(\pi_1,\pi_2)$. If each $\pi_i$ is bounded by $e_i$, one can see that $\tau$ and hence $\sigma$ are bounded by $e_1+e_2$. This requires a case-by-case classification of the representations of $\GL(2)$ along with the explicit isomorphism
\[
    \rho:\GL(2)\times\GL(2)/\Delta\Gm\cong\GSO(2,2).
\]
Under this isomorphism, we have the correspondence of, say, Borel subgroups as
\[
    \rho(\begin{bmatrix}a&\ast\\0&c\end{bmatrix},\begin{bmatrix}a'&\ast\\0&c'\end{bmatrix})
    =\begin{bmatrix}aa'&\ast&\ast&\ast\\0&ac'&\ast&\ast\\0&0&cc'&\ast&\\0&0&0&a'c\end{bmatrix}.
\]
(For this, see the proof of Lemma 8.1 of \cite{Rob01}.) From this the assertion is obvious if both $\pi_1$ and $\pi_2$ are the Langlands quotients of representations induced from the Borel subgroup. Indeed, the spherical case is essentially proven in \cite[Lemma 8.1]{Rob01} in detail. The other cases can be similarly checked, but since the computation is fairly elementary, the detail is left to the reader.

Now if both of $\pi_i$ are infinite dimensional, then since $\sigma$ is unitary, so are $\pi_i$, which implies $e_i<\frac{1}{2}$. If one of $\pi_i$, say $\pi_2$, is one dimensional, then $\pi_2$ is bounded by $\frac{1}{2}$, and also $e_1<\frac{1}{2}$. In either way, $\sigma$ is bounded by $e_1+e_2<1$.

Similarly if $d\neq 1$ and so $\sigma=\tau^\delta$, where $\delta=+$ or $-$ and $\tau=\tau(\pi,\omega)$, then one sees that if $\pi$ is bounded by $e$, then $\sigma$ is bounded by $2e$. (Again see Lemma 8.1 of \cite{Rob01} for the spherical case. The non-spherical case can be derived by considering the explicit correspondence of the parabolic subgroups of $\GL(2,E)$ and $\GSO(V)$. The detail is left to the reader.) And the unitarity implies $e<\frac{1}{2}$ and so $2e<1$.
\end{proof}

Now we are read to prove Theorem \ref{T:main2}

\begin{proof}[Proof of Theorem \ref{T:main2}]
(1). Assume that each local constituent $\sigma_v$ has a non-zero theta lift to $\GSp(4, F_v)$. Let $\sigma'$ be an irreducible constituent of $\sigma|_{\OO(V, \A)}$ (restriction of automorphic forms). It is by now well-known that the standard $L$-function $L^S(s, \sigma')$ of $\sigma'$ is $L^S(s,\pi_1\times\pi_2)$ if $d=1$ and $\tau=\tau(\pi_1,\pi_2)$, or $L^S(s,\pi,\omega^{-1}, Asai)$ if $d\neq 1$ and $\tau=\tau(\pi,\omega)$. In either way, it does not vanish at $s=1$. Also by the above lemma together with Lemma \ref{L:bound_restriction}, we know that each $\sigma'_v$ is bounded by some number less than $1$. (Note that if $V_v$ is anisotropic, clearly $\sigma'_v$ is bounded by less than one.) Moreover by \cite[Lemma 4.2]{Rob96} $\sigma'_v$ has a non-zero theta lift to $\Sp(4, F_v)$. So by the main global non-vanishing result of the previous section, the global theta lift of $\sigma'$ to $\Sp(4, \A)$ is non-vanishing and hence $\Theta_2(\sigma)\neq 0$. The converse is clear.

(2). Let $\sigma'$ be as above. It is again well-known that
\[
     L^S(s, \sigma')=L^S(s+\frac{1}{2},\pi_1\otimes\chi)L^S(s-\frac{1}{2},\pi_1\otimes\chi).
\]
By the same reasoning as above, one sees that $\sigma'_v$ is bounded by some number less than one. Hence just as above we see that $\Theta_2(\sigma)\neq 0$.
\end{proof}

\begin{rmk}
As we mentioned before, part (1) of the theorem has been already proven by Gan and Ichino in their recent preprint \cite{Gan_Ichino} by an entirely different method. For part (2), the same statement is also proven by their method. See (\cite{GT_Siegel} also.) We believe that the converse is also true, and indeed in \cite{GT_Siegel} we show it for the case where the central character of $\pi_1$ is trivial. Also we should mention that for the case where $\pi_1$ has the trivial central character is treated by Schmidt \cite{Schmidt05} in great detail in the context of the Saito-Kurokawa lifting.
\end{rmk}

As the last thing in this subsection, we will prove the following, which essentially shows that our choice of $\tau^+$ and $\tau^-$ coincides with that of Roberts in \cite{Rob01}. Namely we have

\begin{prop}\label{P:local_extension}
Let $F$ be a (not necessarily non-archimedean) local field of $\ch F=0$. Let $\tau\in\Irr(\GSO(V, F))$ be such that $\tau^c\cong\tau$, and so if $d=1$ then $\tau=\tau(\pi,\pi)$, and if $d\neq 1$ then $\tau=\tau(\pi,\omega)$ with $\pi^c=\pi$. Then
\begin{enumerate}[(1)]
\item If $d=1$, then only $\tau^+$ has a non-zero theta lift to $\GSp(4, F)$.
\item If $d\neq 1$, then
\begin{enumerate}[(a)]
\item if $\pi$ is the base change lift of some $\tau_0$ on $\GL(2, F)$ whose central character is $\omega$, both $\pi^+$ and $\pi^-$ have non-zero theta lifts to $\GSp(4,F)$.
\item if $\pi$ is the base change lift of some $\tau_0$ on $\GL(2, F)$ whose central character is $\omega\chi_{E/F}$, then only $\pi^+$ has a non-zero theta lift to $\GSp(4,F)$.
\end{enumerate}
\end{enumerate}
\end{prop}
\begin{proof}
Note that if $F$ is non-archimedean, then part (2) has been proven by Gan and Ichino \cite[Theorem A.11]{Gan_Ichino}. If $F$ is archimedean, then for part (2), $F$ is always $\R$, and the corresponding statement is \cite[Lemma 12]{HST}.

Hence we only have to consider part (1), and so let us assume $d=1$. If $F=\R$ and $V$ is split, then the corresponding statement can be read off from the work of Przebinda \cite{Przebinda}. To be more specific, the restriction of $\tau^+$ to $\OO(2)\times\OO(2)\subseteq\GO(2,2)$ contains the representation of the form $((k)\otimes triv)\oplus(triv\otimes (k))$ for some $k$, where $(k)$ is the representation of $\OO(2)$ with weight $k$ and $triv$ is the trivial representation of $\OO(2)$. Note that $(k)\otimes triv$ (resp. $(triv\otimes (k)$) is what Przebinda means by $\pi_{k,0}^0$ (resp. $\pi_{0,k}^0$) in (2.1.6) of \cite{Przebinda}. Namely in the language of Przebinda $\tau^+$ contains the representations with the lowest degree $K$-types $\pi_{k,0}^0$ and $\pi_{0,k}^0$. (Also note that the restriction of $\tau^-$ to $\OO(2)\times\OO(2)$ contains $((k)\otimes\det)\oplus(\det\otimes (k))$, which is $\pi_{k,0}^1\oplus\pi_{0,k}^1$.) Hence the restriction $\tau^+|_{\OO(2,2)}$ contains a constituent which contains the representation with the lowest degree $K$-type $\pi_{k,0}^0$. Then (3.2.4) and (3.4.31) of \cite{Przebinda} imply that this constituent has a nonzero theta lift to $\Sp(4,\R)$ and hence $\tau^+$ has a nonzero theta lift to $\GSp(4,\R)$. (The author would like to thank A. Paul for checking this. See \cite{Paul2}.) If $F=\C$, this can be similarly read off from \cite[Theorem 2.8]{Ad_Bar95}.

Now the only remaining cases are when $F$ is non-archimedean and $V$ is split, and $F$ is non-archimedean or real and $V$ is anisotropic. In either case we can write $\tau=\tau(\pi,\pi)$, where $\pi$ is a representation of possibly split $\D$. Assume $F$ is non-archimedean. Then we need the following lemma which is due to Roberts.
\begin{lemma}
Let $(\pi, V_{\pi})$ be an irreducible admissible representation of possibly split $\D$. Then $(\omega_{\pi}^{-1}\otimes\pi, V_{\pi})$ is known to be isomorphic to the contragredient $(\pi^\vee, V_{\pi}^\vee)$, and let us fix an isomorphism $R:(\omega_{\pi}^{-1}\otimes\pi, V_{\pi})\rightarrow (\pi^\vee, V_{\pi}^\vee)$. Let $L:V_{\pi}\otimes V_{\pi}\rightarrow \mathbf{1}$ be the composition of $1\otimes R:V_{\pi}\otimes V_{\pi}\rightarrow V_{\pi}\otimes V_{\pi}^\vee$ with the canonical map $V_{\pi}\otimes V_{\pi}^\vee\rightarrow \mathbf{1}$. If $\tau=\tau(\pi,\pi)$, then $\tau^\delta$ has a non-zero theta lift to $\GSp(4)$ if and only if for all $v,w\in V_{\pi}$ we have $L(v\otimes w)=\delta L(w\otimes v)$.
\end{lemma}
Assume $D$ is split and so $\pi$ is a representation of $\GL(2)$. If $\pi$ is finite dimensional and so it is indeed one dimensional, then clearly $\delta=+$ in the above lemma, and hence $\tau^+$ has a non-zero theta lift. So assume $\pi$ is infinite dimensional. Then $V_{\pi}$ can be identified with the space of the Kirillov model. Then the pairing  $L:V_{\pi}\otimes V_{\pi}\rightarrow \mathbf{1}$ has been explicitly described by Jacquet-Langlands. (See the proof of part (i) of Theorem 2.18 in \cite{JL}.) And from the description there, one can easily see that $\delta=+$ in the above lemma. (Let us note that when $\pi$ is spherical, this has been proven by Roberts \cite[Proposition 4.3]{Rob01}.)

Finally assume $D$ is non-split, and so $\pi$ is a finite dimensional representation of $\D$. Also here we assume that $F$ is not necessarily non-archimedean. If $\pi$ is one dimensional, then clearly in the above lemma $\delta=+$. So assume that the dimension of $\pi$ is $>1$. We need to resort to an unfortunate global argument. (We believe that there must be a purely local proof but at this moment we can provide only global one.) First we need the following well-known lemma, which is essentially the Jacquet-Langlands lift via theta lifting.
\begin{lemma}\label{L:JL}
Let $V=D$ be a quaternion algebra over a number field $\F$, and let $\tau=\tau(\pi,\pi)$ be an infinite dimensional cuspidal automorphic representation of $\GSO(D,\A)$. Then there is an extension index $\delta$ such that $\sigma=(\tau,\delta)$ has a non-zero theta lift to $\GSp(2,\A)=\GL(2,\A)$.
\end{lemma}
Now let $\F$ be a number field with a place $v_0$ such that $\F_{v_0}=F$. Pick up another place $v_1$. There exists a quaternion algebra $D$ which ramifies exactly at $v_0$ and $v_1$. One can construct a cuspidal automorphic representation $\Pi$ of $\D(\A)$ such that $\Pi_{v_0}=\pi$ and $\Pi_{v_1}$ is one-dimensional. Note that at all the other $v$, $\Pi_v$ is a representation of $\GL(2,\F_v)$. Now by the above lemma, one sees that there is an extension index $\delta$ so that $\sigma=(\tau,\delta)$ has a non-zero theta lift to $\GSp(2,\A)=\GL(2,\A)$ and hence to $\GSp(4,\A)$. So locally at all $v$, $\sigma_v$ has a non-zero theta lift to $\GSp(4, \F_v)$, and hence by what we have shown so far in this proof, one knows that $\delta(v)=1$ for all $v$ other than $v_0$. But by Proposition \ref{P:global_disconnected}, one also knows that $\underset{v}{\prod}\delta(v)=1$ and hence $\delta(v_0)=1$. This completes the proof.
\end{proof}


\subsection{The theta lift from $\GSp(4)$}


Finally, we consider global theta lifts from $\GSp(4,\A)$. Especially, we consider lifts to the orthogonal group $\GO(V_D,\A)$, where $V_D$ is the 6 dimensional quadratic space given by $V_D=D\oplus\H$, where $D$ is a (possibly split) quaternion algebra and $\H$ is the hyperbolic plane. This pair $(\GSp(4),\GO(V_D))$, both locally and globally, has been studied in great detail by Gan and the author in \cite{GT1} in the context of Shalika period, especially for generic representations of $\GSp(4)$. But here we apply our non-vanishing theorem in the previous section to not necessarily generic ones. To consider this pair, let us first mention that to consider the theta correspondence for the pair $(\GSp(4),\GO(V_D))$, we have to actually consider the pair $(\GSp(4)^+,\GO(V_D))$ with $\GSp(4)^+$ defined by
\[
    \GSp(4)^+=\{g\in\GSp(4):\lambda(g)\in\im(\lambda_D)\},
\]
where $\lambda_D:\GO(V_D)\rightarrow\Gm$ is the similitude character. It can be checked that if $v$ is non-archimedean, then $\GSp(4,F_v)^+=\GSp(4,F_v)$, and if $v$ is archimedean and $D_v$ is split, then $\GSp(4,F_v)^+=\GSp(4,F_v)$. But if $v$ is archimedean and $D_v$ is non-split (so necessarily $v$ is real), then $\GSp(4,F_v)^+$ is the identity component of the real Lie group $\GSp(4,\R)$. So if we let $\Sigma_D$ be the set of places $v$ so that $D_v$ is non-split, and $\Sigma_{D,\infty}\subset\Sigma_D$ the subset of archimedean places in $\Sigma_D$, then we have
\[
    \GSp(4,\A)^+=\prod_{v\in\Sigma_{D,\infty}}\GSp(4,F_v)^+\times{\underset{v\notin\Sigma_{D,\infty}}{{\prod}'}}\GSp(4,F_v).
\]
Note that if $\pi$ is a cuspidal automorphic representation of $\GSp(4,\A)$, then $\pi|_{\GSp(4,\A)^+}$ (restriction of forms) is a non-zero cuspidal automorphic representation of $\GSp(4,\A)^+$. By the global theta lift of $\pi$ to $\GO(V_D,\A)$, we mean the global theta lift of $\pi|_{\GSp(4,\A)^+}$ to $\GO(V_D,\A)$. Also locally at $v\in\Sigma_{D,\infty}$, by the local theta lift of $\pi_v$, we mean the local theta lift of one of the constituents of $\pi_v|_{\GSp(4,\R)^+}$. Indeed, if $\pi_v|_{\GSp(4,\R)^+}$ has two constituents, then both of them have non zero theta lifts to $\GO(V_D,\R)$ and their lifts are isomorphic. The following easily follows from the non-vanishing theorem of the previous section.

\begin{prop}\label{P:lift_to_GO}
Let $\pi$ be an irreducible cuspidal automorphic representation of $\GSp(4,\A)$, and $D$ a division quaternion algebra over $F$. Assume $\pi$ satisfies the following three assumptions:
\begin{enumerate}
\item The standard degree 5 $L$-function $L^S(s,\pi)$ does not vanish at $s=1$.
\item At each place $v$, $\pi_v$ has a non-zero theta lift to $\GO(V_D)$.
\item At each place $v$, $\pi_v$ is bounded by $e_v<1$.
\end{enumerate}
Then the global theta lift of $\pi$ to $\GO(V_D,\A)$ does not vanish.
\end{prop}
\begin{proof}
Let $\pi'$ be an irreducible constituent of $\pi|_{\Sp(4,\A)}$ (restriction of forms). Then the standard $L$-function $L^S(s,\pi')$ is the same as the standard degree 5 $L$-function $L^S(s,\pi)$. Also if each $\pi_v$ has a non-zero local theta lift to $\GO(V_D, F_v)$, then $\pi'_v$ has a non-zero local theta lift to $\OO(V_D, F_v)$. Hence by Theorem \ref{T:main1}, the global theta lift of $\pi'$ to $\OO(V_D,\A)$ does not vanish and so $\pi$ has a non-zero global theta lift to $\GO(V_D,\A)$.
\end{proof}

This allows us to prove Theorem \ref{T:main3}, which is an instance of the Langlands functoriality.

\begin{proof}[Proof of Theorem \ref{T:main3}]
Let us first mention the following fact, which is a known instance of the conservation conjecture on the first occurrence indices of theta lifting. Let $\pi_v$ is an irreducible admissible representation of $\GSp(4,F_v)$ with $v$ non-archimedean, and let
\[
    V_m=\H^m\quad\text{and}\quad V_m^\#=D\oplus\H^{m-2},
\]
and consider the theta lifts $\theta_m(\pi)$ and $\theta^\#_m(\pi)$ to $\GO(V_m)$ and $\GO(V_m^\#)$, respectively. Set
\[
    \begin{cases} m(\pi)=\inf\{m:\theta_m(\pi)\neq 0\}\\
                  m^\#(\pi)=\inf\{m:\theta_m^\#(\pi)\neq 0\}.
    \end{cases}
\]
Then it is known that
\[
    m(\pi)+m^\#(\pi)=6.
\]
The proof of this fact is given in \cite{GT1, GT2}, especially \cite[Section 8]{GT2} and \cite[Section 5]{GT2}.

Now let $\Sigma_{\pi}$ be the set of all (not necessarily) non-archimedean places $v$ so that $\pi_v$ has a non-zero theta lift to $\GO(V_D,F_v)$ where $D$ is the division quaternion algebra. Note that for almost all $v$, $\pi_v$ has a non-zero theta lift to the split $\GO(2,2)$ by \cite[Table 1]{GT4} and hence by the above ``conservation conjecture" for $\GSp(4)$, one can see that $\pi_v$ does not have a non-zero theta lift to $\GO(V_D,F_v)$. So $\Sigma_{\pi}$ is a finite set. Also note that $v_0\in\Sigma_{\pi}$. Now if $|\Sigma_{\pi}|$ is even, one can find a quaternion algebra $D$ over $F$ so that $\Sigma_D=\Sigma_{\pi}$. If $|\Sigma_{\pi}|$ is odd, then one can find a quaternion algebra $D$ over $F$ so that $\Sigma_D=\Sigma_{\pi}\backslash\{v_0\}$. In either case, $D$ is such that at each place $v$, $\pi_v$ has a non-zero theta lift to $\GO(V_D,F_v)$. Hence by the previous proposition, we have a non-zero global theta lift to $\GO(V_D,\A)$.

Assume that this global theta lift is cuspidal, and denote an irreducible constituent by $\sigma$. The restriction $\sigma|_{\GSO(V_D,\A)}$ is identified with a cuspidal automorphic representation $\Pi'$ of $\GL_2(D,\A)$. (See \cite[Introduction]{GT1}.) By \cite{Badulescu, BR}, there is a cuspidal automorphic representation $\Pi$ of $\GL(4,\A)$, namely the global Jacquet-Langlands transfer of $\Pi'$ to $\GL(4,\A)$. By the explicit computations of theta lifts carried out in \cite{GT2, GT4}, one can check that $\Pi$ is the desired functorial lift of $\pi$.

Next assume that the global theta lift of $\pi$ to $\GO(V_D,\A)$ is not cuspidal. Then it has a non-zero theta lift to $\GO(D,\A)$, which is cuspidal. Denote an irreducible constituent by $\sigma$. The restriction $\sigma|_{\GSO(D,\A)}$ is identified with a cuspidal automorphic representation $\pi_1\otimes\pi_2$ of $\D(\A)\times\D(\A)$. If we denote the Jacquet-Langlands lift of $\pi_i$ by $\pi_i^{\JL}$, then the isobaric sum $\pi_1^{\JL}\boxplus\pi_2^{\JL}$ is the desired lift.
\end{proof}

Of course, if $\pi$ is generic, the above transfer has been obtained by Asgari-Shahidi \cite{A-Sh} without any assumption, which can be shown to be the strong lift in \cite{GT2}.\\

Now let us examine how restrictive the assumptions of this theorem are. The $L$-function condition in (1) is supposed to be satisfied for the non-CAP representations, though at this moment the author does not know if there is any method of showing it. (Of course, if $\pi$ is generic, this assumption is known to be satisfied by Shahidi.) Assumption (2) seems to be very small. But it should be mentioned that if $\pi$ corresponds to a Siegel modular form of level $1$, unfortunately this assumption is not satisfied. This can be shown as follows: $\pi_v$ is unramified for all non-archimedean $v$ and hence does not have a non-zero theta lift to $\GO(V_D,F_v)$ for non-split $D$, and at the real place $v$, since $\pi_v$ is non-generic, it cannot have a theta lift to $\GO(3,3)$.

Let us consider assumption (3). This is seemingly the most problematic and indeed crucial for our method. Let $\pi_v$ be an irreducible admissible representation of $\GSp(4)$ which is the Langlands quotient of $\delta_1\times\cdots\times\delta_t\rtimes\tau$. (Of course for $\GSp(4)$, $t$ is at most 2.) Then it is well-known that if $\pi_v$ is unitary, then $e(\delta_1)\leq 1$. (For this assertion, for the non-archimedean case see the table of representations of $\GSp(4)$ by Roberts and Schmidt \cite[Appendix]{table}, and for the archimedean case see \cite[A.3.7]{Przebinda} and citations therein.) Hence the only cases we need to worry about is when $e(\delta_1)=1$. Unfortunately, it is known that such a $\pi_v$ as $e(\delta_1)=1$ does occur as a constituent of a cuspidal representation of $\GSp(4,\A)$. But all the known examples of such cuspidal representations are obtained as theta lifts of characters on some orthogonal groups, which are all CAP representations. (See \cite{Soudry} for the detail of CAP representations.) Accordingly, we believe that such a $\pi_v$ occurs as a constituent of only CAP representations for the following reason. First, let us mention

\begin{lemma}\label{L:generic}
Let $\pi_v$ be an irreducible admissible representation of $\GSp(4,F_v)$ which is the Langlands quotient of the standard module $\delta_1\times\cdots\times\delta_t\rtimes\tau$. If $\pi_v$ is unitary and generic, then $\pi_v$ is bounded by $e_v<1$.
\end{lemma}
\begin{proof}
For the non-archimedean case, one can verify this by looking at the table of representations of $\GSp(4)$ \cite[Appendix]{table}. But in general, one can argue as follows. First of all, as we mentioned above, if $\pi_v$ is unitary, then $e(\delta_1)\leq 1$. Now if $e(\delta_1)=1$, the induced representation $\delta_1\times\cdots\times\delta_t\rtimes\tau$ is reducible. For the non-archimedean case, this can be also seen in the table \cite[Appendix]{table}, and for the archimedean case, the reducibility of the standard modules is found in \cite[p.91]{Przebinda}. (Note that in \cite{Przebinda}, the reducibility is shown only for the real case, but the proof there is completely general, which can be applied to the complex case.) Now by the so called the standard modules conjecture \cite{CSh}, which is known for $\GSp(4)$, we know that $\pi_v$ cannot be generic. Note that for the archimedean case, the standard modules conjecture for the linear groups was proven by Vogan \cite{V} over thirty years ago.
\end{proof}

Now if $\pi$ is an irreducible cuspidal automorphic representation of $\GSp(4, \A)$ which is not CAP, then it is believed that one can find an irreducible cuspidal automorphic representation $\pi'$ of $\GSp(4,\A)$ which is globally generic such that $\pi_v$ and $\pi'_v$ are in the same local $L$-packet at all $v$. Hence if $\pi_v$ is bounded by $e_v<1$, then so does $\pi'_v$. Thus if $\pi$ is not CAP, then each $\pi_v$ should be bounded by $e_v<1$.

Finally, this consideration leads us to our proof of Theorem \ref{T:main4}, which has been already proven by Kudla, Rallis and Soudry \cite[Theorem 8.1]{KRS} with an assumption that the base field $F$ is totally real. However here we give an alternative proof of the theorem which does not require $F$ be totally real and hence remove this unfortunate totally real assumption present in \cite{KRS}.

\begin{proof}[Proof of Theorem \ref{T:main4}]
Note that since $\pi_v$ is generic, by Lemma \ref{L:generic} we know that $\pi_v$ is bounded by $e_v<1$. Assume $v$ is non-archimedean. By \cite[Section 5 and 7]{GT2} one can see that the theta lift of $\pi_v$ to the anisotropic $\GO(D)$ is zero, and hence has a nonzero theta lift to the split $\GO(3,3)$. Next assume that $v$ is complex. Then one sees from \cite[Theorem 2.8]{Ad_Bar95} that $\pi_v$ has a non-zero theta lift to $\GO(3,3)$. (It is not so immediate to see this from this theorem of \cite{Ad_Bar95}. But one can check it case-by-case. Indeed in the notation of \cite{Ad_Bar95}, if $\pi_v$ is the one with the parameters
\[
    \mu_2=(a_1,a_2)\quad\text{and}\quad\nu_2=(b_1,b_2)
\]
then it has a non-zero theta lift to $\GO(3,3)$ whose parameter is
\[
    \mu_1=(a_1,a_2, 0),\quad\nu_1=(b_1,b_2,0)\quad\text{and}\quad\epsilon=1.
\]
Also let us mention that for complex $v$, there is no discrepancy between the isometry theta correspondence and the similitude one, because $\GSp(4,\C)=\C^\times\Sp(4,\C)$ and similarly for the orthogonal groups.) Finally assume $v$ is real. First, for the sake of contradiction, assume $\pi_v$ has a non-zero theta lift to the anisotropic $\GO(D)=\GO(4,0)$. Then any constituent $\pi'_v$ of the restriction $\pi_v|_{\Sp(4,\R)}$ has a non-zero theta lift to either $\OO(4,0)$ or $\OO(0,4)$. But it is well-known that if this happens, then $\pi'_v$ must be a (limit of) holomorphic discrete series, which is never generic. Hence each constituent $\pi'_v$ does not have a non-zero theta lift to $\OO(4,0)$ or $\OO(0,4)$. By \cite[Corollary 4.16]{Paul} one sees that $\pi'_v$ has a non-zero theta lift to $\OO(3,3)$. (Also see the remark right before Theorem 4.8 of \cite{Paul}.) Hence $\pi_v$ has a non-zero theta lift to $\GO(3,3)$. 

Therefore $\pi$ satisfies the assumption of Proposition \ref{P:lift_to_GO} with $D$ split, and hence $\pi$ has a non-zero global theta lift to $\GO(V_D,\A)$ with $D$ split. Assume this theta lift $\Theta_3(V_\pi)$ is in the space of cusp forms, and assume $\sigma$ is an irreducible constituent of $\Theta_3(V_\pi)$. By the main theorem of \cite{GRS} with necessary modifications to the similitude theta lifting, one can see that $\sigma$ has a non-zero generic cuspidal theta lift $\Theta_2(V_\sigma)$ to $\GSp(4,\A)$. Let $\Pi$ be an irreducible constituent of $\Theta_2(V_\sigma)$. Then $\Pi_v\cong\pi_v$ for all $v$, and hence $\Pi\cong\pi$, which implies $\pi$ is globally generic. If $\Theta_3(V_\pi)$ is in the space of cusp forms, by the tower property of theta lifting one sees that $\pi$ has a non-zero theta lift to $\GO(2,2)$. By applying the same argument, one can see that $\pi$ is globally generic.
\end{proof}

The corollary of the theorem (Corollary \ref{C:main}), which is on the multiplicity one theorem for the generic representation of $\GSp(4)$, is now immediate.

\begin{proof}[Proof of Corollary \ref{C:main}]
If $F$ is totally real, this theorem has been proven by Jiang and Soudry \cite{J-S}. Looking at their proof, one notices that the only reason they need this totally real assumption is that they needed the above theorem (Theorem \ref{T:main4}), which had been proven only with the totally real assumption by \cite{KRS}. Hence the above theorem removes this assumption from \cite{J-S}.
\end{proof}


\begin{thebibliography}{999999}


\bibitem[AB]{Ad_Bar95}
J. Adams and D. Barbasch, {\it Reductive dual pair correspondence for complex
groups}, J.\ Funct.\ Anal.\ 132 (1995), 1--42.

\bibitem[ASh]{A-Sh}
M. Asgari and F. Shahidi, {\it Generic transfer from $\rm GSp(4)$ to $\rm GL(4)$}, Compos. Math. 142 (2006), 541--550.

\bibitem[B]{Badulescu}
I. Badulescu, {\it Global Jacquet-Langlands correspondence, multiplicity one and classification of automorphic representations},
with an appendix by N. Grbac, to appear in Invent. Math. (2008).

\bibitem[BR]{BR}
A. I. Badulescu and D. Renard, {\it Unitary Dual of $\GL_n$ at archimedean places and global Jacquet-Langlands corresponde}, preprint.

\bibitem[Ban]{Ban}
D. Ban, {\it Parabolic induction and Jacquet modules of representations of $\OO(2n,F)$}, Glas.\ Mat.\ Ser.\ III 34(54) (1999), 147--185.

\bibitem[BS]{BS91}
S. B\"{o}cherer and R. Schulze-Pillot, {\it Siegel modular
forms and theta series attached to quaternion algebras}, Nagoya
Math.\ J.\ 121(1991), 35--96.

\bibitem[C]{Casselman89}
W. Casselman, {\it Canonical extensions of Harish-Chandra
modules to representations of $G$}, Canad.\ J.\ Math.\ \text{41}
(1989), 385--438.

\bibitem[CSh]{CSh}
W. Casselman and F. Shahidi, {\it On irreducibility of standard modules for generic representations}, Ann. Sci. Ecole Norm. Sup. (4) 31 (1998), 561--589.

\bibitem[Co]{Cogdell04}
J. Cogdell, {\it Lectures on $L$-functions, converse theorems,
and functoriality of $GL(n)$}, in Lectures on Automorphic
$L$-functions, Fields Institute Monographs, AMS (2004), 5--100.

\bibitem[GI]{Gan_Ichino}
W. Gan and A. Ichino : {\it On endoscopy and the refined Gross-Prasad conjecture for $(\SO(5), \SO(4))$}, preprint

\bibitem[GT1]{GT1}
W. Gan and S. Takeda, {\it On Shalika periods and a theorem of Jacquet-Martin}, to appear in Amer. J. Math.

\bibitem[GT2]{GT2}
W. Gan and S. Takeda, {\it The local Langlands conjecture for $\GSp(4)$}, preprint.

\bibitem[GT3]{GT_Siegel}
W. Gan and S. Takeda, {\it The Siegel-Weil formula for the second terms and non-vanishing of theta lifts from orthogonal groups}, preprint.

\bibitem[GT4]{GT4}
W. Gan and S. Takeda, {\it Theta correspondences for $\GSp(4)$}, preprint.

\bibitem[GRS]{GRS}
D. Ginzburg, S. Rallis, and D. Soudry, {\it Periods, poles of $L$-functions and symplectic-orthogonal theta lifts}, J. Reine Angew. Math. 487 (1997), 85--114.

\bibitem[HST]{HST}
M. Harris, D. Soudry, and R. Taylor, {\it l-adic representations associated to modular forms over imaginary quadratic fields I: lifting to $GSp_4(\Q)$}, Invent.\ Math.\ 112 (1993), 377--411.

\bibitem[JL]{JL}
H. Jacquet and R. Langlands, Automorphic Forms on $\GL(2)$. Lect. Notes Math., vol. 114, Berlin Heidelberg New York, Springer (1970).

\bibitem[JS]{J-S}
D. Jiang and D. Soudry, {\it The multiplicity-one theorem for generic automorphic forms of $\GSp(4)$}, Pacific J. Math. 229 (2007), 381--388.

\bibitem[Kd1]{Kudla86}
S. Kudla, {\it On the local theta correspondence}, Invent.\
Math.\ 83 (1986), 229--255.

\bibitem[Kd2]{Kudla_notes}
S. Kudla, Notes on the Local Theta Correspondence, unpublished
notes, available online.

\bibitem[KR1]{Kudla-Rallis90}
S.\ Kudla and S.\ Rallis,\ {\it Poles of Eisenstein series and $L$-functions}. in Festschrift in honor of I. I. Piatetski-Shapiro on the occasion of his sixtieth birthday, Part II (1989), 81--110

\bibitem[KR2]{Kudla-Rallis94}
S.\ Kudla and S.\ Rallis,\ {\it A regularized Siegel-Weil
formula: the first term identity}, Ann.\ Math.\ 140 (1994),
1--80.

\bibitem[KRS]{KRS}
S. Kudla, S. Rallis and D. Soudry, {\it On the degree 5 $L$-function for $\Sp(2)$}, Invent. Math. 107 (1992), 483--541.

\bibitem[M]{Marcus}
D. A. Marcus, Number Fields, Springer,\ (1977).

\bibitem[MVW]{MVW}
C.\ Moeglin, M.-F.\ Vigneras and J.-L.\ Waldspurger : Correspondances de Howe sur un corps $p$-adic, Lecture Notes in Math., vol. 1291, Springer-Verlag, New York, 1987.

\bibitem[P]{Paul}
A. Paul, {\it On the Howe correspondence for
symplectic-orthogonal dual pairs},\ J.\ Functional Analysis,\
228 (2005), 270--310.

\bibitem[P2]{Paul2}
A. Paul, Private correspondence.

\bibitem[PS]{PSP}
D. Prasad and R. Schulze-Pillot, {\it Generalized form of a conjecture of Jacquet and a local consequence},
J. Reine Angew. Math. 616 (2008), 219--236.

\bibitem[Pr]{Przebinda}
T. Przebinda, {\it The oscillator duality correspondndence for the pair $\OO(2,2)$, $\Sp(2,\R)$}, Mem. Am. Math. Sco. 403 (1989).

\bibitem[R1]{Rob96}
B. Roberts, {\it The theta correspondence for similitudes},
Israel J.\ Math.\ 94 (1996), 285--317.

\bibitem[R2]{Rob98}
B. Roberts, {\it Tempered representations and the theta
correspondence},  Canad.\ J.\ Math.\ 50 (1998),
1105--1108.

\bibitem[R3]{Rob99}
B. Roberts, {\it The nonarchimedean theta correspondence for
$\GSp(2)$ and $\GO(4)$}, Trans.\ AMS 351 (1999), 781--811.

\bibitem[R4]{Rob99-2}
B. Roberts, {\it Nonvanishing of global theta lifts from
orthogonal groups},  J.\ Ramanujan Math.\ Soc.\ 14 (1999),
131--194.

\bibitem[R5]{Rob01}
B. Roberts, {\it Global $L$-packets for $\GSp(2)$ and theta
lifts}, Documenta Math.\ \textbf{6} (2001), 247--314.

\bibitem[RS]{table}
B. Roberts and R. Schmidt, Local Newforms for $\GSp(4)$, Lect. Notes Math., Vol. 1918, Berlin Heidelberg New York, Springer (2007).

\bibitem[Sch]{Schmidt05}
R. Schmidt, {\it The Saito-Kurokawa lifting and functoriality}, American J. Math. 127 (2007), 209--240.

\bibitem[S]{Silberger}
A.\ Silberger, {\it Asymptotics and integrability properties for matrix coefficients of admissible representations of reductive $p$-adic groups}, J. Funct. Anal. 45 (1982), 391--402.


\bibitem[So]{Soudry}
D. Soudry, {\it The CAP representations of $\GSp(4,\A)$}, J. reine angew. Math. 383 (1988), 87--108.

\bibitem[Ta]{Tadic95}
M.\ Tadic, {\it Structure arising from induction and Jacquet modules of representations of classical $p$-adic groups}, J.\ of Algebra 177 (1995), 1--33.

\bibitem[Tk]{Takeda}
S. Takeda, {\it Some local-global non-vanishing results for theta lifts from orthogonal groups}, Trans. Amer. Math. Soc. 361 (2009), 5575--5599.

\bibitem[V]{V}
D. A. Vogan, {\it Gelfand-Kirillov dimension for Harish-Chandra modules}, Invent. Math. 48 (1978), 75--98.


\end{thebibliography}
\end{document}